\documentclass[english]{elsarticle}
\usepackage[T1]{fontenc}
\usepackage[latin9]{inputenc}
\usepackage{geometry}
\usepackage{color}
\usepackage{float}
\usepackage{mathrsfs}
\usepackage{mathtools}
\usepackage{amsthm}
\usepackage{graphicx}
\usepackage{amssymb}

\makeatletter

\providecommand{\tabularnewline}{\\}

\theoremstyle{plain}
\newtheorem{thm}{\protect\theoremname}
\theoremstyle{remark}

\newtheorem{remark}{Remark}

\usepackage{babel}

\@ifundefined{showcaptionsetup}{}{%
 \PassOptionsToPackage{caption=false}{subfig}}
\usepackage{subfig}
\makeatother

\usepackage{babel}
\providecommand{\remarkname}{Remark}
\providecommand{\theoremname}{Theorem}

\begin{document}
\begin{frontmatter}{}

\title{From networked SIS model to the Gompertz function}

\author{Ernesto Estrada}

\address{Institute for Cross-Disciplinary Physics and Complex Systems (IFISC,
UIB-CSIC),\\ Campus Universitat de les Illes Balears E-07122, Palma
de Mallorca, Spain.\\ \textbf{E-mail}: estrada@ifisc.uib-csic.es}

\author{Paolo Bartesaghi}

\address{Department of Statistics and Quantitative Methods,\\
	University of Milano - Bicocca, Via Bicocca degli Arcimboldi 8, 20126, Milano, Italy.\\
	\textbf{E-mail}: paolo.bartesaghi@unimib.it}

\begin{abstract}
\textcolor{black}{The Gompertz function is one of the most widely
used models in the description of growth processes in many different
fields. We obtain a networked version of the Gompertz function as
a worst-case scenario for the exact solution to the SIS model on networks.
This function is shown to be asymptotically equivalent to the classical
scalar Gompertz function for sufficiently large times. It proves to
be very effective both as an approximate solution of the networked
SIS equation within a wide range of the parameters involved and as
a fitting curve for the most diverse empirical data. As an instance,
we perform some computational experiments, applying this function
to the analysis of two real networks of sexual contacts. The numerical results highlight the analogies and the differences between the exact description provided by the SIS
model and the upper bound solution proposed here, observing how the
latter amplifies some empirically observed behaviors such as the
presence of multiple and successive peaks in the contagion curve. }

\textcolor{black}{\medskip{}
}

\textbf{\textcolor{black}{AMS Subject Classification}}\textcolor{black}{{:
92D39; 05C82, 37N25}}
\end{abstract}
\end{frontmatter}{}

\section{Introduction}

The Gompertz function was first introduced by Benjamin Gompertz in
1825 to describe human mortality curves \citep{gompertz1825xxiv}.
Formally, a Gompertz curve can be defined as a non-negative real valued
function defined on the open interval $0<t<\infty$ \citep{oshima2020modified}:

\begin{equation}
f\left(t\right)=P\exp\left[-Qe^{-Rt}\right],\label{Gompertz}
\end{equation}
where $P,Q,R>0$. It can be proved that it is the solution of the
differential equation of the form (see \citep{oshima2020modified}
for the formal proof):

\begin{equation}
\dot{f}\left(t\right)=\frac{df(t)}{dt}=QRe^{-Rt}f(t).
\end{equation}

A point of inflection for the Gompertz curve is the ordered pair $\left(\dfrac{\log Q}{R},\dfrac{P}{e}\right).$
This curve is appropriate to describe many physical, biological and
man-made processes, which are characterized by a rapid growth in the
early stages and slower decrease in the late ones. For instance, Finch
and Pike \citep{finch1996maximum} have examined maximum life span
predictions obtained with the Gompertz mortality rate model and observed
that in mammals and birds there is a good agreement on the maximum
life span predicted by the model and the ones reported for local populations.
In biology, the Gompertz model is frequently used to describe the
growth of bacteria and cancer cells. In the last case, many studies
are reported in the mathematical biology literature (see, for instance,
\citep{tumor_1,tumor_2,tumor_3}). The justification for using the
Gompertzian models in modeling tumor growth was provided by Frenzen
and Murray \citep{frenzen1986cell}. They proposed two related maturity-time
cell kinetics model mechanisms, which at large time evolve to a Gompertz
form. More recently, Karin et al. \citep{karin2019senescent} found
that senescent cell turnover slows with age which gives an explanation
for the Gompertz law. It is plausible that the ubiquity of the Gompertz
model is due to its ``diffusive'' nature. Gutierrez-Jaimez et al.
\citep{gutierrez2007new} have proposed a new Gompertz-type diffusion
process which allows that bounded sigmoidal growth patterns are modeled
by time-continuous variables. Another diffusion-like approach was
proposed by Li et al. \citep{li2018dynamic}.

Recently, it has been discovered that the Gompertz curve is appropriate
to describe different growing processes related to the pandemic of
COVID-19\footnote{COVID-19 is the acronym for COronaVIrus Disease 2019}.
For instance, Ramirez-Torres et al. \citep{ramirez2021new} have used
it for estimating the number of unreported cases of COVID-19 in a
region/country. Conde-Gutierrez et al. \citep{conde2021comparison}
have used the Gompertz function for predicting the dynamics of deaths
from the pandemics, while Berihuete et al. \citep{berihuete2021bayesian}
used Bayesian approaches implementing the Gompertz function to forecast
the evolution of the contagious disease and evaluate the success of
particular policies in reducing infections. Mandujano Valle \citep{valle2020predicting}
also used this function to estimate the total number of infected and
deaths by COVID-19 in Brazil and two Brazilian States (Rio de Janeiro
and S\~{a}o Paulo). In another work, Ohnichi et al. \citep{ohnishi2020universality}
demonstrated the existence of a universal scaling behavior for the
number of cases of COVID-19 in 11 countries using the Gompertz function.

The similarities and differences between the Gompertz and the logistic
model have been the topic of much debate \citep{comparison_1,comparison_2,comparison_3,comparison_4}.
In general, it has been observed that (i) the lack of symmetry of
the Gompertz curve, (ii) the fact that its point of inflection occurs
earlier than in the logistic, and (iii) that the carrying capacity
is reached relatively earlier than in the logistic curve, are advantages
in fitting growth data \citep{comparison_4}. All in all, it can be
concluded that the Gompertz curve is more useful to fit data of some
growth processes than the logistic curve. Additionally, there is evidence
on the superiority of the Gompertz vs. the logistic curve in fitting
empirical data about the disease progression of pathosystems for more
than 100 diseases in plants \citep{comparison_3}, as well as in reproducing
COVID-19 data \citep{achterberg2020comparing}.


\begin{figure}[H]
	\centering \includegraphics[width=0.45\textwidth]{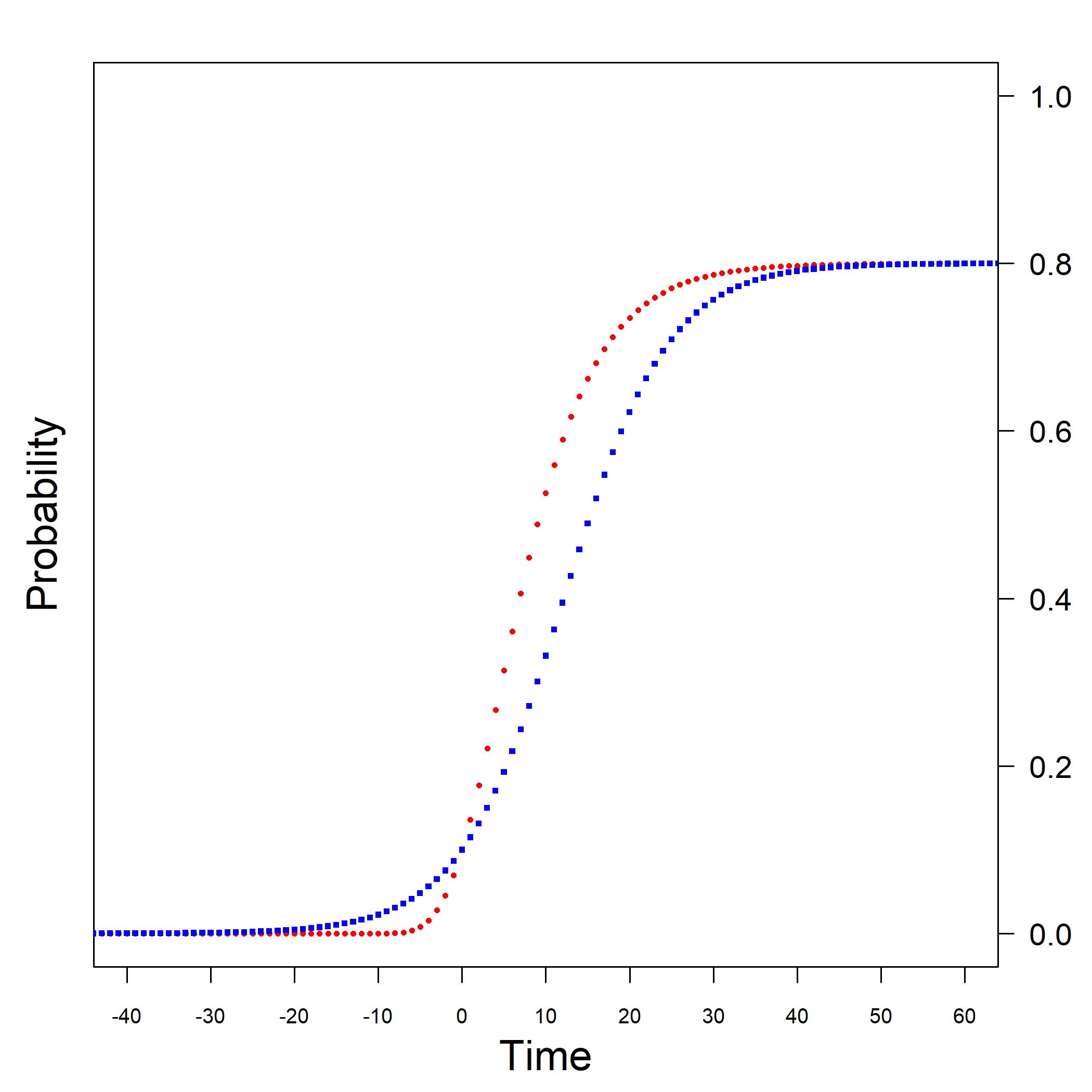}
	\includegraphics[width=0.45\textwidth]{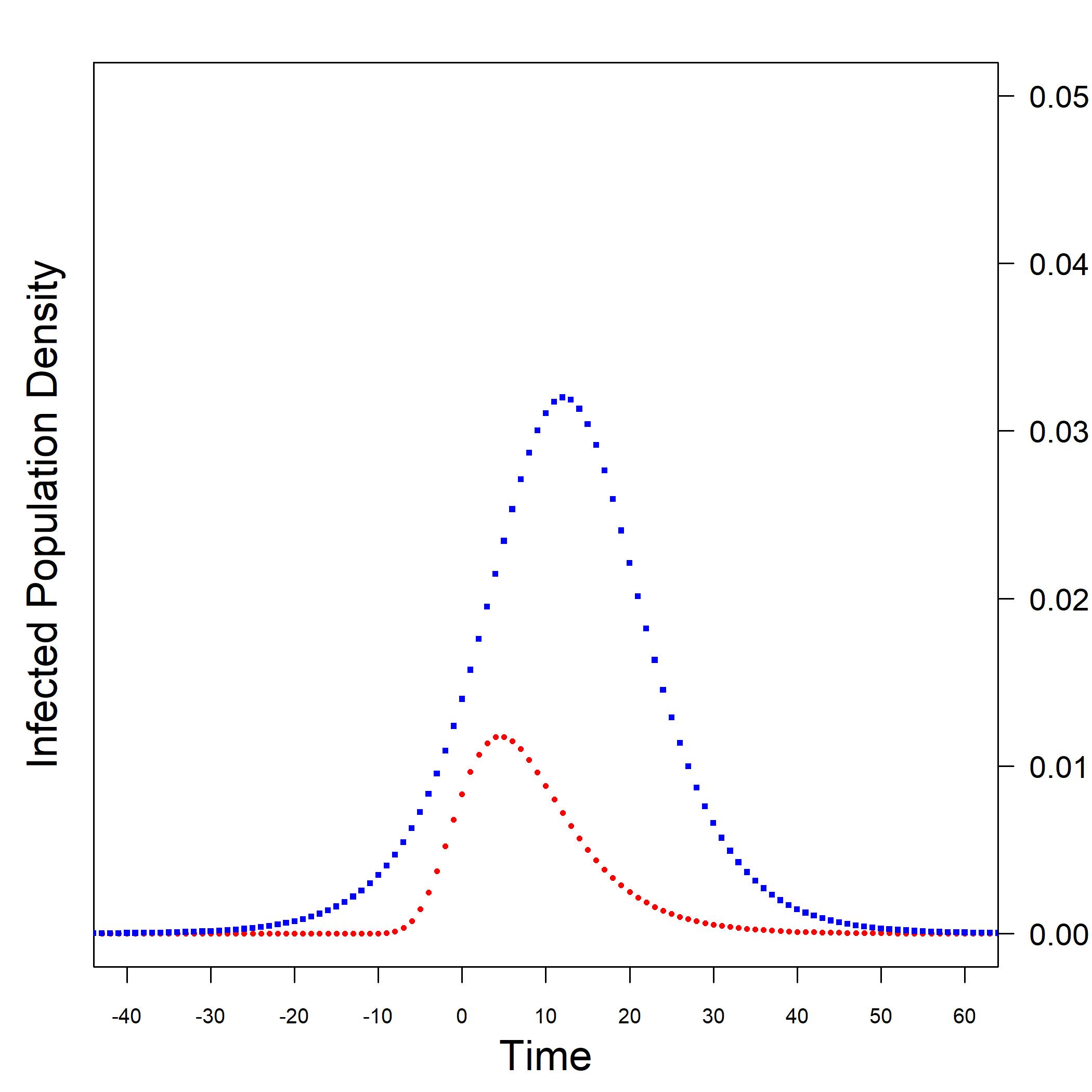}
\caption{Differences between SIS model (blue square dots plots) and Gompertz model (red circular dots plots). Left panel illustrates the cumulative distribution of the probability to find an infected individual at time $t$ (blue plot is the logistic curve and red plot is the Gompertz curve). Right panel illustrates the infected population density, i.e. the number of individuals who become infected exactly at time $t$, in both models. See the text for explanation and parameters values.}
	\label{fig0} 
\end{figure}

\textcolor{black}{The standard logistic growth curve can be interpreted
as the solution of the Susceptible-Infected-Susceptible (SIS) model,
the epidemiological model characterized by the possibility of re-infection
after recovery (see, for instance, \citep{Hui2021} for an up-to-date introduction to the SIS model). In Fig. \ref{fig0}, we illustrate the difference
between SIS model and Gompertz model. Blue square dots curves in both
boxes refer to the SIS model and red circular dots curves to the Gompertz
model. Curves in the left box represent the normalized cumulative
number of cases $N(t)$ as a function of time $t$, whereas curves
in the right box represent the first derivative $dN(t)/dt$. Curves on the left
can be interpreted as the cumulative probabilities to find an infected
individual at time $t$ and curves on the right as the probabilities
that new individuals get infected at time $t$, or equivalently as
the infected population density. All the plots refer to the same set
of parameters values. The initial fraction of infected individuals
or, more in general, the initial population size is set to $N_{0}=0.1$
and the asymptotic fraction of the infected individuals or, more in
general, the final carrying capacity is set to $N_{\infty}=0.8$.
In the SIS model, the value of the infection rate is $\beta=0.2$
and the value of the recovering rate is $\gamma=0.04$. Note that,
with these values, we have $N_{\infty}=1-\frac{\gamma}{\beta}=0.8$.
In the Gompertz model described by Eq. (\ref{Gompertz}), these parameters
correspond to $P=N_{\infty}$, $R=\beta\left(1-\frac{\gamma}{\beta}\right)$ and $Q=\log\frac{N_{\infty}}{N_{0}}$,
so that, in particular, the value for the damping rate in the Gompertz
model is $R=0.16$. Blue lines show the typical symmetrical behavior
of the SIS model and of the logistic curve with respect to its inflection
point; red lines show the typical asymmetrical behavior and the long
tail of the Gompertz model with respect to its inflection point. In
fact, the inflection point of the logistic curve is in $N_{\infty}/2=0.4$ at time $t=\frac{1}{R}\log\left(\frac{N_{\infty}}{N_{0}}-1\right)\approx 12.16$, whereas the Gompertz curve is asymmetrical, with an inflection point lower than $N_{\infty}/2$, namely $N=P/e\approx0.3$ at time $t=\frac{{\log}Q}{R}=\frac{\log\log8}{0.16}\approx 4.6$. Finally, all plots have been extended to negative times to highlight the complete behavior of the curves.}

\section{A worst-case scenario SIS model }

Let $G=\left(V,E\right)$ be a simple connected graph on $n$ vertices
(nodes). We will use indistinctly graph or network to refer to $G$.
The networked SIS model is described by the following $n$ equations:

\begin{equation}
\dot{x}_{i}(t)=\beta[1-x_{i}(t)]\sum_{j=1}^{n}A_{ij}x_{j}(t)-\gamma x_{i}(t),\label{eq1}
\end{equation}
where $x_{i}(t)$ represents the probability that node $i$ is infected at time $t$ and $x_{i}(0)=x_{0i}=p$, $\forall i=1,\dots,n$, for some $0<p<1$, the initial probability of being infected, equal for all nodes. Elements $A_{ij}$ are the corresponding entries of the adjacency matrix $\mathbf{A}$ of $G$. In matrix-vector form, Eq. (\ref{eq1}) becomes:

\begin{equation}
\dot{{\bf x}}(t)=\beta\left[{\bf I}-{\rm diag}({\bf x}(t))\right]{\bf A}\,{\bf x}(t)-\gamma{\bf x}(t),\label{eq2}
\end{equation}
with ${\bf x}=\left[{x_1}, {x_2}, \dots , {x_n}\right]\in {\mathbb R}^{n}$ and ${\bf x}(0)={\bf x}_{0}=p\,{\bf u}$, where ${\bf I}$ is the
identity matrix of the corresponding order and ${\bf u}$ the vector whose components are all equal to $1$. In what follows, we write in general ${\bf v}\preceq {\bf w}$ for ${\bf v}, {\bf w}\in {\mathbb R}^{n}$ if ${v_i}\leq {w_i},\ \forall i=1,\dots , n$.

Let us now rewrite the networked SIS model on the basis of the information
content $I_{i}(t)$ that a given node $i$ is not infected at time
$t$. The information content (also known as surprisal, self-information,
or Shannon information) is given by

\begin{equation}
I_{i}(t)=-\log[1-x_{i}(t)],\quad\forall i=1,\dots,n,\label{eq3}
\end{equation}
where $s_{i}\left(t\right)=1-x_{i}(t)$ is the probability of not
being infected at time $t$, i.e., of being susceptible. The information
content SIS model (IC-SIS) is now written as

\begin{equation}
\dot{I}_{i}(t)=\beta\sum_{j=1}^{n}A_{ij}\left(1-e^{-I_{j}\left(t\right)}\right)-\gamma\sum_{j=1}^{n}\delta_{ij}\left(e^{I_{j}\left(t\right)}-1\right),\label{eq4}
\end{equation}
with $\delta_{ij}$ being the Kronecker delta and $I_{i}(0)=-\log(1-p)=-\log q$,
where $q=1-p$.

Let us consider the function $f(I)=1-e^{-I}$, which is increasingly
concave, so that we can write

\[
f(I)<f(I_{0})+f'(I_{0})(I-I_{0})=e^{-I_{0}}I+1-e^{-I_{0}}(I_{0}+1).
\]

Analogously, because $g(I)=e^{I}-1$ is increasingly convex we have

\[
g(I)>g(I_{0})+g'(I_{0})(I-I_{0})=e^{I_{0}}I-1-e^{I_{0}}(I_{0}-1).
\]

Therefore, we have

\[
\dot{I}_{i}(t)<\beta\sum_{j=1}^{n}A_{ij}\left[e^{-I_{0}}I_{j}+1-e^{-I_{0}}(I_{0}+1)\right]-\gamma\sum_{j=1}^{n}\delta_{ij}\left[e^{I_{0}}I_{j}-1-e^{I_{0}}(I_{0}-1)\right].
\]

Let us call the right-hand-side part of this equation the worst-case scenario
IC-SIS as it represents a clear upper bound to the surprisal that
a given node $i$ is not infected at time $t$. Let us designate it
by $\dot{\hat{I}}_{i}(t)$ and taking into account that $x_{0}=p$,
$I_{0}=-\log q$, $e^{-I_{0}}=q$ and $e^{I_{0}}=1/q$, we can write
it as

\[
\dot{\hat{I}}_{i}(t)=\beta q\sum_{j=1}^{n}A_{ij}\hat{I}_{j}(t)-\frac{\gamma}{q}\sum_{j=1}^{n}\delta_{ij}\hat{I}_{j}(t)+\beta(p+q\log q)\sum_{j=1}^{n}A_{ij}-\gamma\,\frac{p+\log q}{q},
\]
where $\hat{I}_{i}(t)$ is the worst-case scenario surprisal that
the node $i$ is not infected at time $t$. When written in matrix-vector
form

\begin{equation}
\dot{\hat{\mathcal{{\bf \mathscr{I}}}}}(t)=\left(\beta q{\bf A}-\frac{\gamma}{q}{\bf I}\right){\hat{{\bf \mathscr{I}}}}(t)+\left[\beta(p+q\log q){\bf A}-\frac{\gamma}{q}(p+\log q){\bf I}\right]{\bf u},\label{eq7}
\end{equation}
it is easy to realize that this equation has the linear form

\begin{equation}
\dot{\hat{\mathcal{{\bf \mathscr{I}}}}}(t)={\bf B}{\hat{{\bf \mathscr{I}}}}+{\bf b},\label{eq8}
\end{equation}
where ${\bf B}\coloneqq q\beta{\bf A}-\frac{\gamma}{q}{\bf I}$ and
${\bf b}\coloneqq\left[(p+q\log q)\beta{\bf A}-(p+\log q)\frac{\gamma}{q}{\bf I}\right]{\bf u}$.

As a way of comparison, we will also consider the ``standard''
linearization of the SIS model around the point ${\bf 0}$:

\begin{equation}
\dot{\mathbf{x}}\left(t\right)=\left(\beta\mathbf{A}-\gamma\mathbf{I}\right)\mathbf{x}\left(t\right), \label{eq9}
\end{equation}
with initial condition $\mathbf{x}_{0}$, whose solution is ${\mathbf{\bar x}}\left(t\right)=\exp\left[\left(\beta\mathbf{A}-\gamma\mathbf{I}\right)t\right]\mathbf{x}_{0}$.
It is easy to realize that this solution is exponentially unstable. We now formally prove that the solution $\mathbf{\hat{x}}\left(t\right)\coloneqq 1-\exp\big[-{\hat{{\bf \mathscr{I}}}}\big]$
to the worst-case scenario SIS model is an upper bound -- therefore
a worst-case scenario solution -- to the exact SIS model solution,
i.e. $\mathbf{x}\left(t\right)\preceq\mathbf{\hat{x}}\left(t\right),$ and that it is a lower bound to the divergent solution of the linearized model, i.e. $\mathbf{\hat{x}}\left(t\right)\preceq {\mathbf{\bar{x}}}\left(t\right)$.

\begin{thm}
Let $\mathbf{x}\left(t\right)$, $\mathbf{\hat{x}}\left(t\right)$ and $\mathbf{\bar{x}}\left(t\right)$ be, respectively, the solution of the exact SIS model in Eq. (\ref{eq2}), of the worst-case scenario IC-SIS model obtained from Eq. (\ref{eq7}) and of the linearized SIS model in Eq. (\ref{eq9}), with the same initial conditions: $\mathbf{x}\left(0\right)=\mathbf{\hat{x}}\left(0\right)={\mathbf{\bar{x}}}\left(0\right)=p\mathbf{u}$.
Then, $\forall t\geq0$, 
\begin{equation}
\mathbf{x}\left(t\right)\preceq\mathbf{\hat{x}}\left(t\right)\preceq {\mathbf{\bar{x}}}\left(t\right).\label{eq10}
\end{equation}
where the second inequality holds if $\beta_{e}\coloneqq \frac{\beta}{\gamma}>\frac{1}{q^{2}k_{\rm min}}$, with $k_{\rm min}$ minimum degree of the graph.
\end{thm}
\begin{proof}
We first prove that $\mathbf{x}\left(t\right)\preceq\mathbf{\hat{x}}\left(t\right)$.
By the method of variation of parameters, the solution of the worst-case scenario IC-SIS model is
\begin{equation}
\hat{{\bf \mathscr{I}}}(t)=e^{{\bf B}t}\,\mathscr{I}_{0}+\left(e^{{\bf B}t}-{\bf I}\right){\bf B}^{-1}{\bf b},\label{eq11}
\end{equation}
which is equivalent to
\begin{equation}
\hat{{\bf \mathscr{I}}}(t)=e^{{\bf B}t}\left[{\bf B}^{-1}{\bf b}-\log q\,{\bf u}\right]-{\bf B}^{-1}{\bf b}.\label{eq12}
\end{equation}
Since $I_{i}(0)=\hat{I}_{i}(0)=-\log q$ and $\dot{I}_{i}(t)<\dot{\hat{I}}_{i}(t),\ \forall t\geq 0$, $\forall i=1,\dots, n$, following Lemma A.1 by Lee et al. \citep{lee2019transient}, p. 11,
\begin{equation}
	\hat{x}_{i}(t)=1-e^{-\hat{I}_{i}(0)}.\label{eq13}
\end{equation}
is an upper bound for the original SIS network model solution ${x}_{i}(t), \forall t\geq 0$, $\forall i=1,\dots, n$, and this proves the first inequality.
We now prove that $\mathbf{\hat{x}}\left(t\right)\preceq {\mathbf{\bar{x}}}\left(t\right)$.
To this purpose, let us assume that $q\beta k_{\rm min}-\frac{\gamma}{q}>0$, which implies $q\beta \lambda_{1}-\frac{\gamma}{q}>0$, where $\lambda_1$ is the eigenvalue corresponding to the dominant eigenvector of $\bf A$. Under this condition, since $q<1$, we have $\beta \lambda_1 \geq \beta k_{\rm min}> \frac{\gamma}{q^2}> \frac{\gamma}{q} > {\gamma}$ and both the solutions of the upper bound problem $\mathbf{\hat{x}}\left(t\right)$ and of the linearized problem ${\mathbf{\bar{x}}}\left(t\right)$ are increasing functions in time. Again, by Lemma A.1 in \citep{lee2019transient}, since the initial conditions are the same for both processes, i.e., $\mathbf{\hat{x}}\left(0\right)={\mathbf{\bar{x}}}\left(0\right)=\mathbf{x}_{0}=p\mathbf{u}$,
it is enough to prove that
\begin{equation}
\dfrac{d\mathbf{\hat{x}}\left(t\right)}{dt}\preceq\dfrac{d{\mathbf{\bar{x}}}\left(t\right)}{dt}
\label{eq14}
\end{equation}
for all $t\geq0$. Being $\mathbf{\hat{x}}\left(t\right)=1-e^{-\mathbf{\hat{\mathscr{I}}}}$,
then we have
\begin{equation}
\dfrac{d\mathbf{\hat{x}}\left(t\right)}{dt}=e^{-\mathbf{\hat{\mathscr{I}}}}\dfrac{d\mathbf{\hat{\mathscr{I}}}\left(t\right)}{dt}\preceq\dfrac{d\hat{\mathbf{\mathscr{I}}}\left(t\right)}{dt}\label{eq15}
\end{equation}
for all $t\geq0$, where the inequality in Eq. (\ref{eq15}) follows from $e^{-{\hat{I}_{i}}}<1$
for all $\hat{I}_{i}\in[0,\infty)$. By Eq. (\ref{eq12}) we have
\[
\dfrac{d\mathbf{\hat{\mathscr{I}}}\left(t\right)}{dt}=e^{\mathbf{B}t}\mathbf{B}\left[\mathbf{B}^{-1}{\bf b}-\log q\,{\bf u}\right]=e^{\mathbf{B}t}\left[{\bf b}-\log q\,\mathbf{B}{\bf u}\right]
\]
where $\mathbf{B}=\beta q\mathbf{A}-\frac{\gamma}{q}\mathbf{I}$ and
${\bf b}=\left[\left(p+q\log q\right)\beta\mathbf{A}-\left(p+\log q\right)\frac{\gamma}{q}\mathbf{I}\right]{\bf u}$.
Since
\[
{\bf b}-\log q\,\mathbf{B}{\bf u}=p\left(\beta\mathbf{A}-\frac{\gamma}{q}\mathbf{I}\right){\bf u}
\]
we have
\begin{equation}
\dfrac{d\mathbf{\hat{\mathscr{I}}}\left(t\right)}{dt}=e^{\left(q\beta {\bf A}-\frac{\gamma}{q}{\bf I}\right)t}\left[p\left(\beta{\bf A}-\frac{\gamma}{q}{\bf I}\right){\bf u}\right]\preceq e^{\left(\beta{\bf A}-\gamma{\bf I}\right)t}\left[p\left(\beta{\bf A}-\gamma{\bf I}\right){\bf u}\right]=\dfrac{d{\mathbf{\bar{x}}}\left(t\right)}{dt}.\label{eq16}
\end{equation}
The last inequality can be justified in the following way. Being $\bf A$ and $\bf I$ commuting matrices, Eq. (\ref{eq16}) is equivalent to $e^{q \beta {\bf A}t}e^{-\frac{\gamma}{q}{\bf I}t}\left(\beta{\bf A}-\frac{\gamma}{q}{\bf I}\right){\bf u}\preceq e^{\beta{\bf A}t}e^{-{\gamma}{\bf I}t}\left(\beta{\bf A}-\gamma{\bf I}\right){\bf u}$ and, hence, to the inequality
\begin{equation*}
e^{q \beta {\bf A}t}\left(\beta{\bf A}-\frac{\gamma}{q}{\bf I}\right){\bf u}\preceq e^{\frac{p}{q} \gamma t} e^{\beta{\bf A}t}\left(\beta{\bf A}-\gamma{\bf I}\right){\bf u}.
\end{equation*}
Moreover, $q \beta {\bf A}\succeq 0$ and $\beta {\bf A}\succeq 0$, with $q \beta {\bf A}\preceq  \beta {\bf A}$ element by element, such that
\begin{equation*}
	{\bf 0}\preceq \left(\beta{\bf A}-\frac{\gamma}{q}{\bf I}\right){\bf u}\preceq \left(\beta{\bf A}-{\gamma}{\bf I}\right){\bf u}
\end{equation*}
where $\bf 0$ is the null vector, since we assume $\beta k_{\rm min}-\frac{\gamma}{q}>0$. Therefore, we have $e^{q \beta {\bf A}t}\left(\beta{\bf A}-\frac{\gamma}{q}{\bf I}\right){\bf u}\preceq e^{\beta {\bf A}t}\left(\beta{\bf A}-\gamma{\bf I}\right){\bf u}$, for all $t\geq 0$ and, being $e^{\beta {\bf A}t}\left(\beta{\bf A}-\gamma{\bf I}\right){\bf u}\preceq e^{\frac{p}{q} \gamma t} e^{\beta{\bf A}t}\left(\beta{\bf A}-\gamma{\bf I}\right){\bf u}$, for all $t\geq 0$, all the components of the vector in the left hand side of Eq. (\ref{eq16}) are less than or equal to the corresponding components of the vector in the right hand side. Finally Eq. (\ref{eq15}) and Eq. (\ref{eq16}) imply Eq. (\ref{eq14}) and this ends the proof.
\end{proof}
\medskip{}

\begin{remark}
If $\gamma=0$, we have ${\bf B}=q\beta{\bf A}$ and ${\bf b}=(p+q\log q)\beta{\bf A}{\bf u}$
so that ${\bf B}^{-1}{\bf b}=\left(\frac{p}{q}+\log q\right){\bf u}$.
Solution in Eq. (\ref{eq12}) reduces to

\[
\hat{{\bf \mathscr{I}}}(t)=\frac{p}{q}e^{q\beta{\bf A}t}{\bf u}-\left(\frac{p}{q}+\log q\right){\bf u},
\]
and it is equal to the solution for the SI Model by Lee et al. (see \citep{lee2019transient} and \citep{lee2019transientProceedings}). Let us observe that if $t=0$: $\hat{{\bf \mathscr{I}}}(0)=-\log q\,{\bf u}$
and ${{\bf \hat{x}}}(0)=p$, as expected. Moreover, for $t\to+\infty$,
$\hat{{\bf \mathscr{I}}}(t)\to+\infty$ and ${{\bf \hat{x}}}(t)\to{\bf u}$. 
\end{remark}

\begin{remark}
If $\beta=0$, ${\bf B}=-\frac{\gamma}{q}{\bf I}$ and ${\bf b}=-(p+\log q)\frac{\gamma}{q}{\bf u}$
so that ${\bf B}^{-1}{\bf b}=(p+\log q){\bf u}$. Solution in Eq. (\ref{eq12})
reduces to

\[
\hat{{\bf \mathscr{I}}}(t)=p\left(e^{-\frac{\gamma}{q}t}-1\right){\bf u}-\log q\,{\bf u}.
\]
Let us observe that for $t\to+\infty$, $\hat{{\bf \mathscr{I}}}(t)\to -(p+\log q){\bf u}$
and $\hat{{\bf x}}(t)\to(1-qe^{p}){\bf u}$. This bound doesn't converge
to ${\bf x}^{\star}={\bf 0}$ as $t\to+\infty$ but to $(1-qe^{p}){\bf u}$.
Observe that $0<1-qe^{p}<p$, as expected, and that $1-qe^{p}\sim p^2$
as $p\to0$. 
\end{remark}

\medskip{}
\begin{remark}
Let us consider $\beta\neq0$ and $\gamma\neq0$. The exponential
term in Eq. (\ref{eq12}) may be written as

\[
e^{{\bf B}t}=e^{(q\beta{\bf A}-\frac{\gamma}{q}{\bf I})t}=e^{(q\beta{\bf M}{\bf \Lambda}{\bf M}^{T}-\frac{\gamma}{q}{\bf I})t}=e^{{\bf M}(q\beta{\bf \Lambda}-\frac{\gamma}{q}{\bf I}){\bf M}^{T}t}={\bf M}e^{(q\beta{\bf \Lambda}-\frac{\gamma}{q}{\bf I})t}{\bf M}^{T},
\]
where ${\bf \Lambda}$ is the diagonal matrix of the eigenvalues of
${\bf A}$ and ${\bf M}$ is the orthogonal matrix whose columns are
the eigenvectors of ${\bf A}$. As $t$ grows to $+\infty$, the diagonal
exponential terms $e^{(q\beta\lambda_{i}-\frac{\gamma}{q})t}$ grows
to $+\infty$ if $(q\beta\lambda_{i}-\frac{\gamma}{q})>0$. In particular,
if $(q\beta\lambda_{1}-\frac{\gamma}{q})<0$, no one of these terms
grows to $+\infty$ and the epidemic decays. Thus, we can identify
a threshold given by the following condition

\begin{equation}
\beta_{e}=\frac{\beta}{\gamma}<\frac{1}{q^{2}\lambda_{1}}\coloneqq \tau,\label{eq17}
\end{equation}
and we can identify $\tau=\frac{1}{q^{2}\lambda_{1}}$ as the threshold
of the worst-case scenario solution. Then, if $\beta_{e}<\tau$ the epidemic decays;
if $\beta_{e}>\tau$ the epidemic grows. Let us observe that, as $\lambda_{1}$ increases (and so does the average degree in the
network), condition in Eq. (\ref{eq17}) becomes stricter and the spread of epidemics
is easier; moreover, in general, $\tau$ is greater than $1/\lambda_{1}$ and, as $p$ decreases, it approaches the lower bound $1/\lambda_{1}$.
\end{remark}

\subsection{Gompertz-like function from SIS}
Let us start by defining a matrix ${\bf D}$ as
\begin{equation}
	{\bf D}\coloneqq{\bf I}-q^{2}\beta_{e}{\bf A},\label{eq18a}
\end{equation}
so that ${\bf B}=-\frac{\gamma}{q}{\bf D}$. Then we can rewrite the
vector ${\bf B}^{-1}{\bf b}-\log q\,{\bf u}$ appearing in Eq. (\ref{eq12})
in the following way:
\[
{\bf B}^{-1}{\bf b}-\log q\,{\bf u}=\frac{p}{q}\left[{\bf I}-p({\bf I}-q^{2}\beta_{e}{\bf A})^{-1}\right]{\bf u}=\frac{p}{q}\left[{\bf I}-p{\bf D}^{-1}\right]{\bf u}.
\]
In this way, the solution of the worst-case scenario IC-SIS model becomes
\begin{equation}
	\hat{{\bf \mathscr{I}}}(t)=\frac{p}{q}\left[e^{-\frac{\gamma}{q}{\bf D}t}-{\bf I}\right]\left[{\bf I}-p{\bf D}^{-1}\right]{\bf u}-\log q\,{\bf u}.\label{eq19}
\end{equation}
We are now in a position to present the main result of the current work. 
\begin{thm}
	\label{main theorem_new}Let $G$ be a graph in which there is a IC-SIS
	dynamics taking place on its nodes and edges. Let $\psi_{\nu}\left(i\right)$
	be the $i$-th entry of the $\nu$-th (orthonormalized) eigenvector
	of $\mathbf{A}$, associated with the eigenvalue $\lambda_{\nu}$, with $\lambda_{n}\leq \dots \leq \lambda_{2}\leq\lambda_{1}$.
	Then, under the condition $\beta_e <q\tau$, the worst-case scenario probability $\tilde{s}_{i}(t)$ that a node $i$ is susceptible of getting infected for time $t\rightarrow\infty$ is given by a Gompertz function of the form
	\begin{equation}
		\tilde{s}_{i}(t) =P_{i}\exp\left[-Q_{i}e^{-Rt}\right]\label{Gompertz1}
	\end{equation}
	where $P_{i}=q\exp \left[ \sum_{\nu=1}^{n}\alpha_{\nu}\zeta_{\nu}(i)\right]$, $Q_{i}=\alpha_{1} \zeta_{1}(i)$ and $R=\frac{\gamma}{q}(1-q^{2}\beta_{e}\lambda_{1})$ with $\alpha_{\nu}\coloneqq\frac{p-pq\beta_{e}\lambda_{\nu}}{1-q^{2}\beta_{e}\lambda_{\nu}}$ and $\zeta_{\nu}(i)\coloneqq\psi_{\nu}(i)\sum_{j=1}^{n}\psi_{\nu}(j)$, $\forall \nu=1, \dots, n$.
	\label{maintheorem}
\end{thm}
\begin{proof}
Let's start from Eq. (\ref{eq19})
\begin{equation*}
	\hat{{\bf \mathscr{I}}}(t)=\frac{p}{q}\left[e^{-\frac{\gamma}{q}{\bf D}t}-{\bf I}\right]\left[{\bf I}-p{\bf D}^{-1}\right]{\bf u}-\log q\,{\bf u}
\end{equation*}
Since ${\bf D}={\bf I}-q^{2}\beta_{e}{\bf A}$ is symmetric and there exists an orthogonal matrix ${\bf M}$ whose columns are the eigenvectors of ${\bf A}$ and, consequently, of ${\bf D}$ such that
\begin{equation*}
{\bf D}={\bf M}\, {\rm diag}(1-q^{2}\beta_{e}\lambda_{\nu})\, {\bf M}^{T}
\end{equation*}
we have
\begin{equation*}
\begin{split}
\hat{{\bf \mathscr{I}}}(t)
&=\frac{p}{q}{\bf M}\, {\rm diag}\left[ \left(e^{-\frac{\gamma}{q}(1-q^{2}\beta_{e}\lambda_{\nu})t}-1\right)\,  \left(1-\frac{p}{1-q^{2}\beta_{e}\lambda_{\nu}}\right)\right] {\bf M}^{T}\,{\bf u} -\log q\,{\bf u}\\
&=\frac{p}{q}{\bf M}\, {\rm diag}\left[ \frac{q-q^{2}\beta_{e}\lambda_{\nu}}{1-q^{2}\beta_{e}\lambda_{\nu}} \left(e^{-\frac{\gamma}{q}(1-q^{2}\beta_{e}\lambda_{\nu})t}-1\right)\,  \right] {\bf M}^{T}\,{\bf u} -\log q\,{\bf u}\\
&={\bf M}\, \Lambda(t) {\bf M}^{T}\,{\bf u} -\log q\,{\bf u}
\end{split}
\end{equation*}
where $\Lambda(t)={\rm diag}\left[\eta_{\nu}(t)\right]$ and $\eta_{\nu}(t)=\alpha_{\nu} \left(e^{-\frac{\gamma}{q}(1-q^{2}\beta_{e}\lambda_{\nu})t}-1\right)$ with $\alpha_{\nu}\coloneqq\frac{p-pq\beta_{e}\lambda_{\nu}}{1-q^{2}\beta_{e}\lambda_{\nu}}$. Now, since $\hat{{\bf s}}(t)=e^{-\hat{{\bf \mathscr{I}}}(t)}$:
\begin{equation*}
\begin{split}
\hat{s}_{i}(t)
&=q\exp \left[ -\left( {\bf M}\, \Lambda(t) {\bf M}^{T}\,{\bf u}\right)_{i}  \right]\\
&=q\exp \left[ -\sum_{j=1}^{n}\sum_{\nu=1}^{n}\psi_{\nu}\left(i\right)\psi_{\nu}\left(j\right)\eta_{\nu}(t) \right]\\
&=q\exp \left[ - \sum_{\nu=1}^{n}\zeta_{\nu}(i)\eta_{\nu}(t) \right]\\
\end{split}
\end{equation*}
where $\zeta_{\nu}(i)\coloneqq \psi_{\nu}(i)\sum_{j=1}^{n}\psi_{\nu}(j)$. Now, in order to describe the asymptotic behavior of $\hat{s}_{i}(t)$, let us consider separately the following three cases:
\begin{enumerate}
	\item $\beta_e <q\tau$, that is $1-q\beta_e\lambda_1 >0$, which implies $1-q^2\beta_e\lambda_1 >0$:
\begin{equation}
\begin{split}
\hat{s}_{i}(t)
&= q\exp \left[ - \sum_{\nu=1}^{n}\zeta_{\nu}(i)\eta_{\nu}(t) \right]\\
&\sim q\exp \left[ -\left( \zeta_{1}(i)\alpha_{1} e^{-\frac{\gamma}{q}(1-q^{2}\beta_{e}\lambda_{1})t}-\sum_{\nu=1}^{n}\zeta_{\nu}(i)\alpha_{\nu}\right) \right]\\
&=q\exp \left[  \sum_{\nu=1}^{n}\zeta_{\nu}(i)\alpha_{\nu}\right] \exp \left[- \zeta_{1}(i)\alpha_{1} e^{-\frac{\gamma}{q}(1-q^{2}\beta_{e}\lambda_{1})t} \right]\coloneqq \tilde{s}_{i}(t)\\
\end{split}\label{eq30}
\end{equation}
In particular, for $t\to +\infty$
\begin{equation*}
\hat{s}_{i}(t) \to q\exp \left[ \sum_{\nu=1}^{n}\zeta_{\nu}(i)\alpha_{\nu}\right]
\end{equation*}
The last term in Eq. (\ref{eq30}) can be interpreted as a Gompertz function. Indeed, it is
\begin{equation}
\tilde{s}_{i}(t)= P_{i}\exp\left[-Q_{i}e^{-Rt}\right]
\end{equation}
with $P_{i}=q\exp \left[ \sum_{\nu=1}^{n}\alpha_{\nu}\zeta_{\nu}(i)\right]$, $Q_{i}=\alpha_{1}\zeta_{1}(i)$ and $R=\frac{\gamma}{q}(1-q^{2}\beta_{e}\lambda_{1})$. Let us observe that $\zeta_{1}(i)=\psi_{1}(i)\sum_{j=1}^{n}\psi_{1}(j)$ is positive $\forall i=1,\dots, n$ as the eigenvector associated to the dominant eigenvalue $\lambda_1$ has components all with the same sign. Moreover, in order to get also $Q_{i}>0$, the further condition $1-q\beta_{e}\lambda_{1}>0$, equivalent to $\beta_e <q\tau$, must be satisfied. Let us notice that this condition is stricter than $1-q^2\beta_{e}\lambda_{1}>0$, which is equivalent to $\beta_e<\tau$.
\item $q\tau <\beta_e < \tau$: in this case Eq. (\ref{eq30}) still holds but it cannot be considered a Gompertz function since $\alpha_1 <0$.
\item $\beta_e >\tau$, that is $1-q^2\beta_e\lambda_1 <0$:
\begin{equation}
	\begin{split}
		\hat{s}_{i}(t)
		&= q\exp \left[ -\sum_{\nu=1}^{n}\zeta_{\nu}(i)\eta_{\nu}(t) \right]\\
		&\sim q \exp \left[-\alpha_{1} \zeta_{1}(i) e^{-\frac{\gamma}{q}(1-q^{2}\beta_{e}\lambda_{1})t} \right]\\
	\end{split}\label{eq31a}
\end{equation}
where again $\zeta_{1}(i)>0$ and also $\alpha_{1}=\frac{p-pq\beta_{e}\lambda_{1}}{1-q^{2}\beta_{e}\lambda_{1}}>0$ since now $1-q\beta_e\lambda_1 <0$. Therefore, in this case, we have as expected $\hat{s}_{i}(t)\to 0$ for $t \to + \infty$.
\end{enumerate}
\end{proof}
\begin{remark}
According to Theorem \ref{maintheorem}, when the time is sufficiently large, the solution of the worst-case scenario IC-SIS on a network is a Gompertz function of the type: $f\left(t\right)=P\exp\left[{-Qe^{-Rt}}\right]$, under the condition that the effective infectivity rate is lower than $q\tau =\frac{1}{q\lambda_1}$. This is an important result because, as we have seen in the Fig. \ref{fig0}, the main differences between the SIS model and the Gompertz model occurs not at early times of the disease propagation but when the time is sufficiently large. Additionally, it is also at large time when the two models developed by Frenzen and Murray \citep{frenzen1986cell} coincide with a Gompertz curve, which would open some interesting avenues for future exploration of these models in relation to the current one. 
\end{remark}
\begin{remark}
Let us notice that, according to Eq. (\ref{Gompertz1}), the probability density can be written as
\begin{equation}
	\frac{d\tilde{s}_{i}(t)}{dt}=\tilde{s}_{i}(t)\left[Q_{i}Re^{-Rt}\right]
\end{equation}
and the point of inflection of the Gompertz curve is then given by
\begin{equation}
t=\frac{\log Q_{i}}{R}=\frac{q}{\gamma(1-q^2\beta_e \lambda_1)}\log [\alpha_1 \zeta_1(i)]
\end{equation}
This inflection point, which is the time at which the probability density of recovering from infection is maximum, grows with the logarithm of $\zeta_1(i)$ and $\zeta_1(i)$ is proportional to the eigenvector centrality of node $i$. This means that the more central is the node, in terms of eigenvector centrality, the more this maximum is delayed in time and the slower will be the healing process for that node.
\end{remark}
\begin{remark}
Finally, we conclude by observing that Eq. (\ref{eq19}) can be given the following interesting expression.
When the effective infectivity rate is below the threshold $\tau$, that is if it is satisfied the condition $0<q^{2}\beta_{e}<\frac{1}{\lambda_{1}}$, the inverse ${\bf D}^{-1}=\left[{\bf I}-q^{2}\beta_{e}{\bf A}\right]^{-1}$ can be considered a matrix resolvent and ${\bf D}^{-1}{\bf u}=\left[{\bf I}-q^{2}\beta_{e}{\bf A}\right]^{-1}{\bf u}={\bf u}+{\bf c}$,
where ${\bf c}=\left[\left(\mathbf{I}-q^{2}\beta_{e}\mathbf{A}\right)^{-1}-\mathbf{I}\right]\mathbf{u}$, for $0<q^{2}\beta_{e}<\frac{1}{\lambda_1}$, is the vector whose entries are the so-called Katz centrality indices of the nodes in the network \citep{katz1953}. This index belongs to a family of network descriptors which counts the number of walks in the graph involving the corresponding node, giving more weight to the shorter than to the longer ones  \citep{Estrada2010SIAM}.
Therefore we have
\[
\left[{\bf I}-p{\bf D}^{-1}\right]{\bf u}={\bf u}-p{\bf D}^{-1}{\bf u}={\bf u}-p\,{\bf u}-p\,{\bf c}=q\,{\bf u}-p\,{\bf c}.
\]
and the upper bound solution of the IC-SIS model in Eq. (\ref{eq19}) may be expressed in terms of the Katz centralities of the nodes as
\begin{equation}
	\hat{{\bf \mathscr{I}}}(t)=p\left[e^{-\frac{\gamma}{q}{\bf D}t}-{\bf I}\right]\left[{\bf u}-\frac{p}{q}\,{\bf c}\right]-\log q\,{\bf u}.\label{eq20}
\end{equation}
\end{remark}

\subsection{Computational experiments}

We now present some numerical results from computational experiments carried out on two real-world sexual contact networks. The rationale for choosing networks of this type lies in the fact that they typically host contagion processes that can be described using the SIS model. The spreading of many different epidemics of global interest, like influenza, Ebola, Zika, Chikungunya, and others, has been analyzed by applying this model. For instance, Juher et al. \citep{juher2017} studied in details the Syphilis transmission characteristics and disease evolution on a large sexual network in San Francisco and Zhao et al. \citep{zhao2019} investigate, in the same framework, the temporal patterns and transmission potential of ZIKA virus in eight Brazilian states, over the same period.

Here we focus on two specific sexual contact networks. Both networks
are connected, unweighted and undirected. The first corresponds to
$m=83$ heterosexual contacts between $n=82$ individuals ($35$ males
and $47$ females). In this network there is available information
about the sex of the individuals. The network was studied in details
in \citep{wylie2001patterns} where a subset of individuals were infected
with \textit{Chlamydia Trachomatis} and/or \textit{Neisseria Gonorrhoeae}
in Manitoba, Canada. We consider here the giant connected component
of this network, which spans two broad geographic areas, northern
Manitoba and Winnipeg. The second network of sexual contacts represents
$m=266$ homosexual contacts between $n=250$ individuals. It is known
that 236 individuals (94.4\%) were men, of whom 219 were homosexual,
and 14 (5.6\%) were women. Of the 184 men and 13 women for whom age
was known, mean age (at mid-early period) was 29.5 years for men and
25.8 years for women. The sex of every individual was not reported
in the dataset available. This network has been previously studied
in \citep{potterat2002risk} to explain the modest HIV/AIDS burden
reported from Colorado Springs at the earliest days of the epidemic.
Accordingly, the risk networks of susceptible individuals have been
insufficiently cohesive to sustain substantial endogenous HIV transmission
\citep{potterat2002risk}. An illustration of both networks studied
here is given in Fig. \ref{Networks}.

\begin{figure}[h]
\begin{centering}
\subfloat[]{\begin{centering}
\includegraphics[width=1\textwidth]{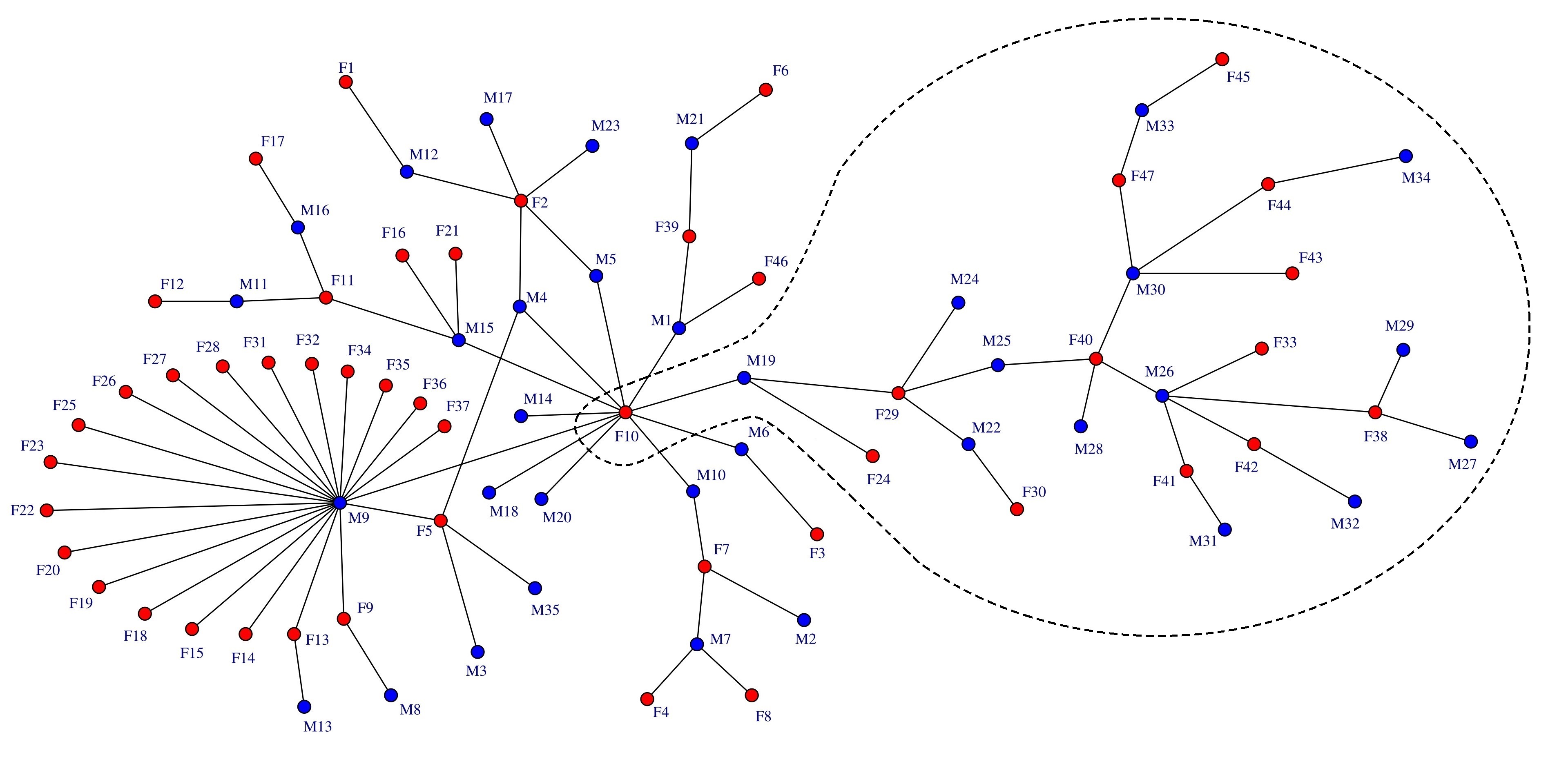} 
\par\end{centering}
}
\par\end{centering}
\begin{centering}
\subfloat[]{\includegraphics[width=1\textwidth]{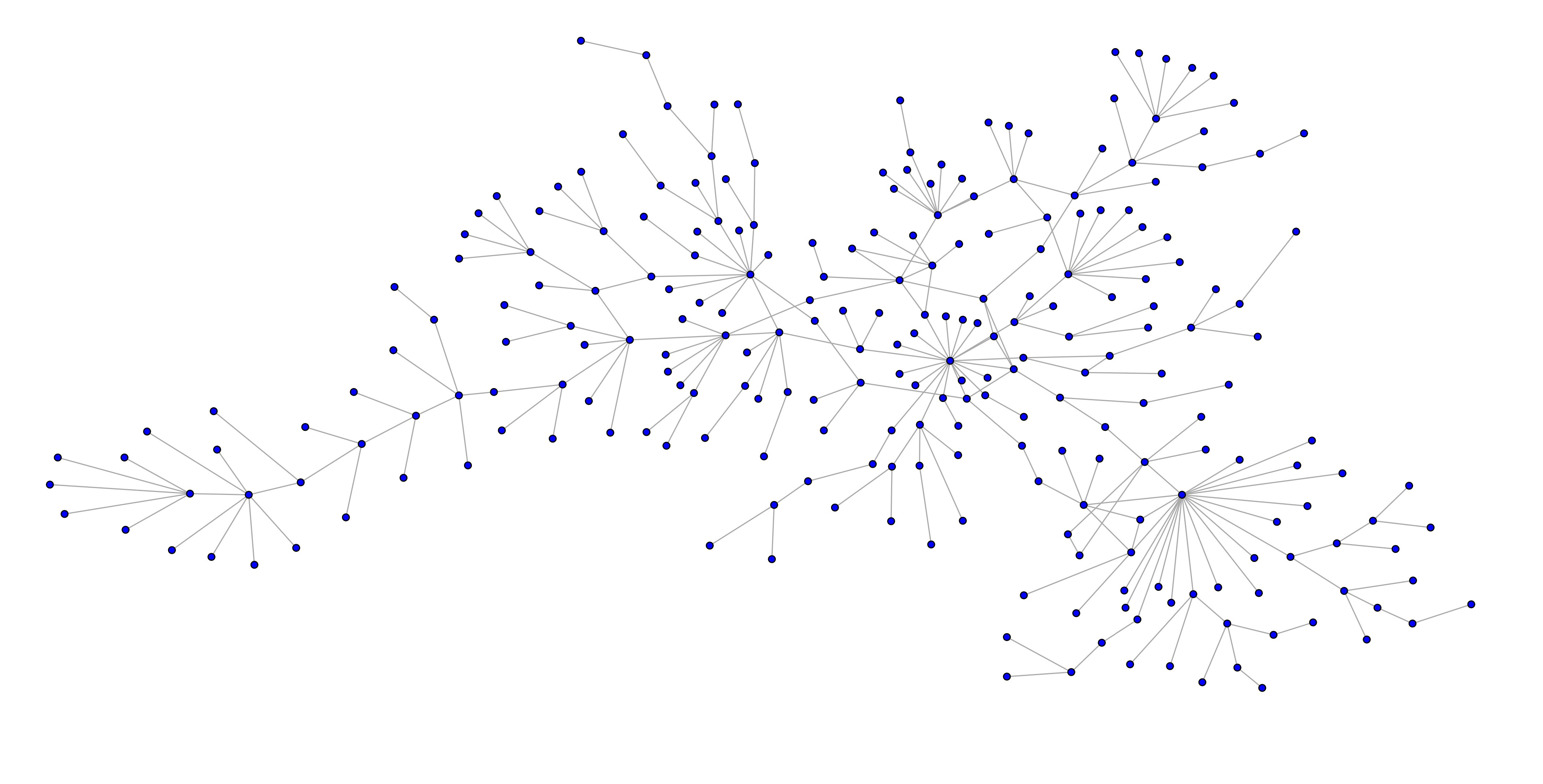}

}
\par\end{centering}
\caption{Illustration of the networks of (a) heterosexual contacts (NHeC) and (b) homosexual contacts (NHoC) studied in the paper. In the NHeC, the nodes are labelled by the sex of the individuals, M for male and F for female. The block of individuals residing in Winnipeg are sketched with a broken line. The rest are from Northern Manitoba.}

\label{Networks} 
\end{figure}

In Table \ref{comparison} we compare some of the general parameters
of both networks (see \citep{estrada2012structure} for definitions
and examples): (i) edge density $\delta=2m/n\left(n-1\right)$ where
$m$ is the number of edges, (ii) average shortest path length $\bar{l}$,
(iii) average Watts-Strogatz clustering coefficient $\bar{C}$ (see
also \citep{watts1998collective}), (iv) degree heterogeneity $\rho$
(see also \citep{estrada2010quantifying,estrada2019degree}), (v)
spectral radius of the adjacency matrix $\lambda_{1}$ (see also \citep{stevanovic2014spectral}),
and (vi) degree assortativity $r$ (see also \citep{newman2003mixing}).
As can be seen, the two networks are very similar to each other, except
that the network of heterosexual contacts (NHeC) is bipartite and thus it contains no
odd cycles, which implies that the clustering coefficient is zero,
while the network of homosexual contacts (NHoC) is not bipartite and it has clustering
different from zero. Both networks are poorly dense, although the NHeC is more densely connected than the NHoC. The degree heterogeneity of both networks is relatively small, their spectral radii are very close to each other, and both networks are degree disassortative indicating a preference of high degree nodes of being connected with
low degree ones.

\begin{table}[h]
\begin{centering}
\begin{tabular}{|c|c|c|}
\hline 
property  & Heterosexual  & Homosexual\tabularnewline
\hline 
\hline 
$\delta$  & 0.025  & 0.0085\tabularnewline
\hline 
$\bar{l}$  & 5.387  & 8.327\tabularnewline
\hline 
$\bar{C}$  & 0  & 0.0298\tabularnewline
\hline 
$\rho$  & 0.0154  & 0.0159\tabularnewline
\hline 
$\lambda_{1}$  & 4.7333  & 4.8454\tabularnewline
\hline 
$r$  & -0.0458  & -0.2782\tabularnewline
\hline 
\end{tabular}
\par\end{centering}
\caption{A few topological parameters of the two networks of sexual contacts
analyzed: $\delta$ is the edge density; $\bar{l}$ is the average
path length; $\bar{C}$ is the average Watts-Strogatz clustering coefficient;
$\rho$ is the degree heterogeneity; $\lambda_{1}$ is the spectral
radius of the adjacency matrix; $r$ is the degree assortativity.}

\label{comparison} 
\end{table}

\textcolor{black}{The main goal of this section is to compare the
networked SIS model and the Gompertz-like solution produced as an
upper bound of the SIS model. The SIS model is used for sexually transmitted
diseases where individuals can get repeatedly infected. Eames and Keeling \citep{eames2002,keeling2005} provide
some specific values for $\beta_{e}$ for both }\textit{\textcolor{black}{Chlamydia}}\textcolor{black}{{}
and }\textit{\textcolor{black}{Gonorrhoea}}\textcolor{black}{{}, according to the different models used to describe the virus
spreading in the same network we analyze here. For the purposes of
our numerical analysis it is enough to keep a mean value equal to
$\beta_{e}=1.5$. Therefore, we will consider a disease propagating
through these two sexual networks, for which we use the following
parameters: $\beta=0.03$, $\gamma=0.02$ and $\beta_{e}=1.5$. In
both cases we consider the same probability $p,$ which for the NHeC is $1/82$ and for the NHoC is $3/250$. This choice
is equivalent to a single initially infected individual in the NHeC; an equal initial probability is then adopted in the NHoC to compare the time evolution of the epidemics in both.}

\subsection{Time evolution of the cumulative probability}

We start our analysis by considering the time evolution of the cumulative
number of infected individuals in both networks. In Fig. \ref{Time_evolution}
we illustrate the results of the simulations with the parameters given
before using the exact solution of the SIS model (panel (a)) as well
as with the Gompertz-like solution of the worst-case scenario (panel
(b)). As can be seen in panel (a) the evolution of the infection predicted
by the SIS model is practically identical for both networks up to
approximately $t=40$. From that time the propagation of the infection
in the NHoC is faster than in the NHeC.
In the steady state the number of infected individuals in both networks
is below 65\% with a slightly bigger percentage in the NHoC than
in the NHeC. In the case of the Gompertz-like model
the growing process occurs identically for both networks for a longer
time than in the SIS model, i.e., up to $t=50$ (see panel (b)). Also,
after this time the infection grows faster in the NHoC
reaching the 100\% of infected individuals at $t=90$ while the NHeC
reaches the saturation at $t=110$.

To gain more insights about the importance of these differences we
consider some data observed empirically in \citep{wylie2001patterns}.
According to this observation the branch formed by individuals $F6-M21-F39-M1-F46-F10$
(branch 1) turned out to be completely infected while the branch of
individuals $F45-M33-F47-M30-F43-F4$0 (branch 2) was not completely
infected at the same time \textcolor{black}{(only 66.7\% of individuals
infected)}. Notice that both branches contains the same number of
nodes, as well as the same proportion of males/females. As can be
seen in Fig. \ref{Time_evolution} (panel (c)) the exact solution of
the SIS model predicts a faster growth of the infection in branch
1, but when this branch reaches its maximum, which is not 100\% of
infected nodes, the difference with branch 2 is insignificant. This
indicates that using SIS we cannot observe branch 1 completely infected
and branch 2 only partially infected. The Gompertz approach (see panel
(d)) also predicts a faster growth of the infection on branch 1, but
with bigger differences than in the case reported by SIS. For instance,
at $t=80$, SIS predicts 45.7\% of infection in branch 1 and 35.4\%
in branch 2, while Gompertz predicts 90.9\% for branch 1 and only
56.1\% for branch 2. More importantly, at $t=100$ all the nodes in
branch 1 are infected, while the probability of being infected for
nodes in branch 2 is below 80\%. Therefore, in this case we can reproduce,
even without fitting the parameters $\beta$ and $\gamma$, what has
been observed empirical in the real-world evolution of a sexually
transmitted disease in this network.

%
%
%
%
%

\begin{figure}[h]
	\subfloat[]{\includegraphics[width=0.45\textwidth]{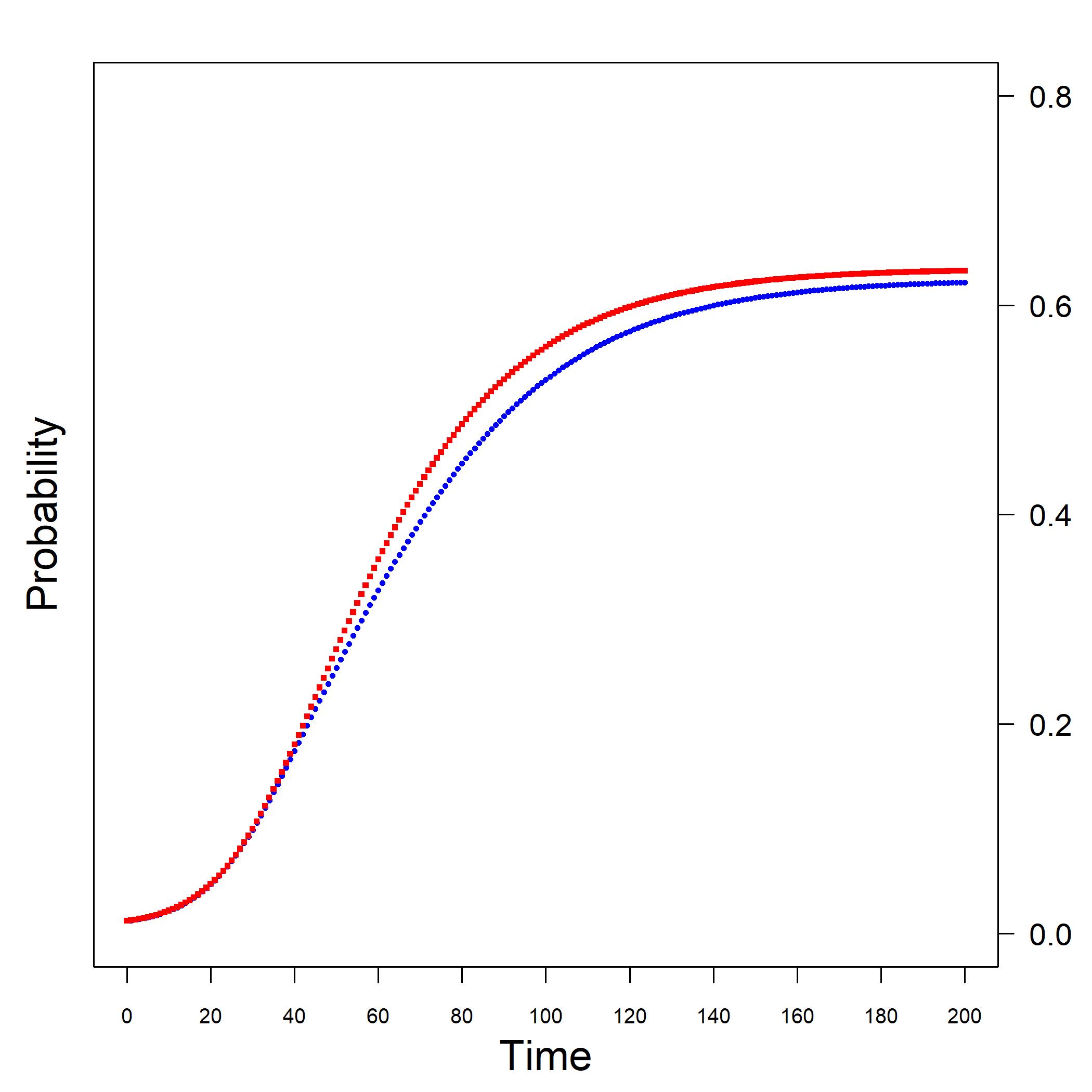}
		
	}\subfloat[]{\includegraphics[width=0.45\textwidth]{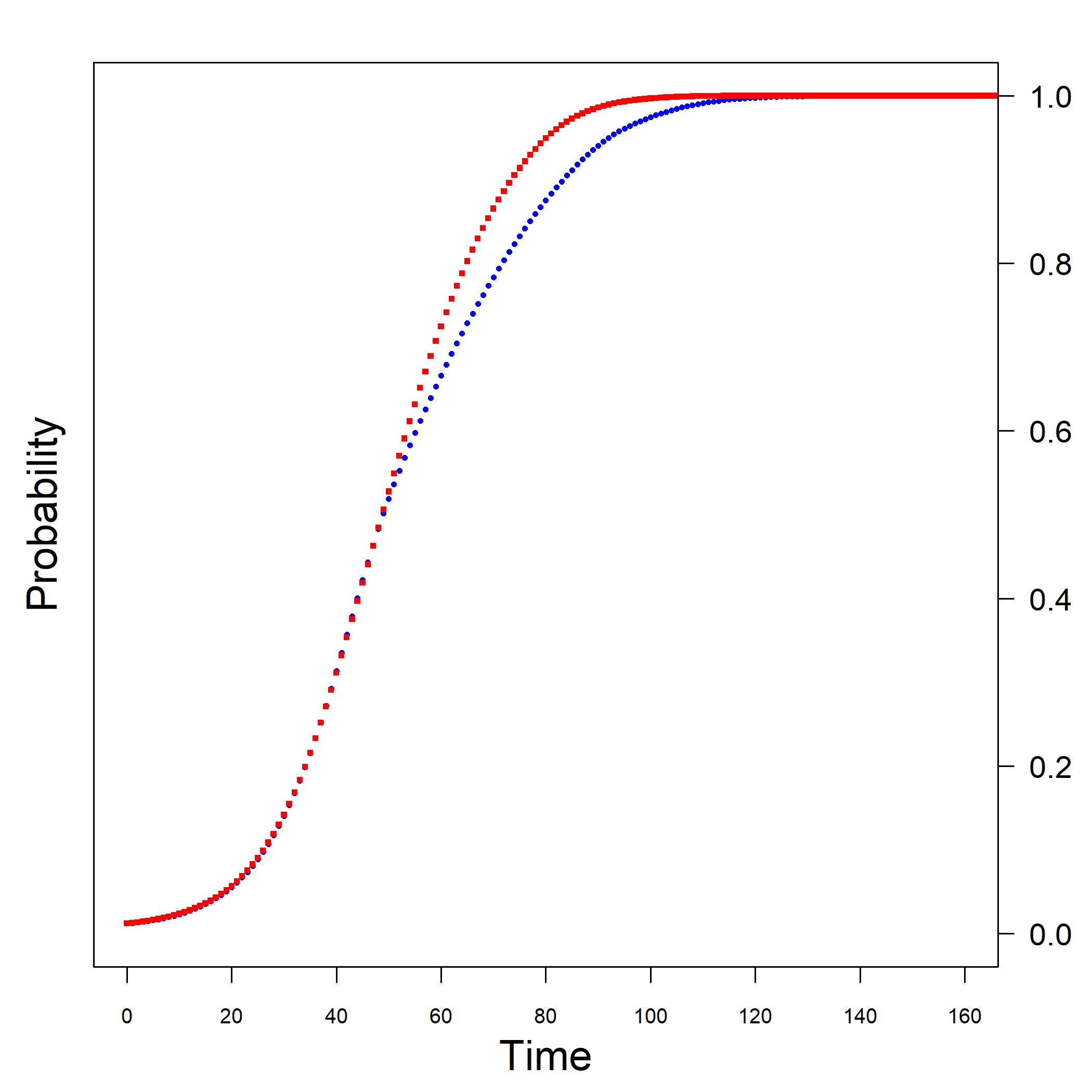}
		
	}
	
	\subfloat[]{\includegraphics[width=0.45\textwidth]{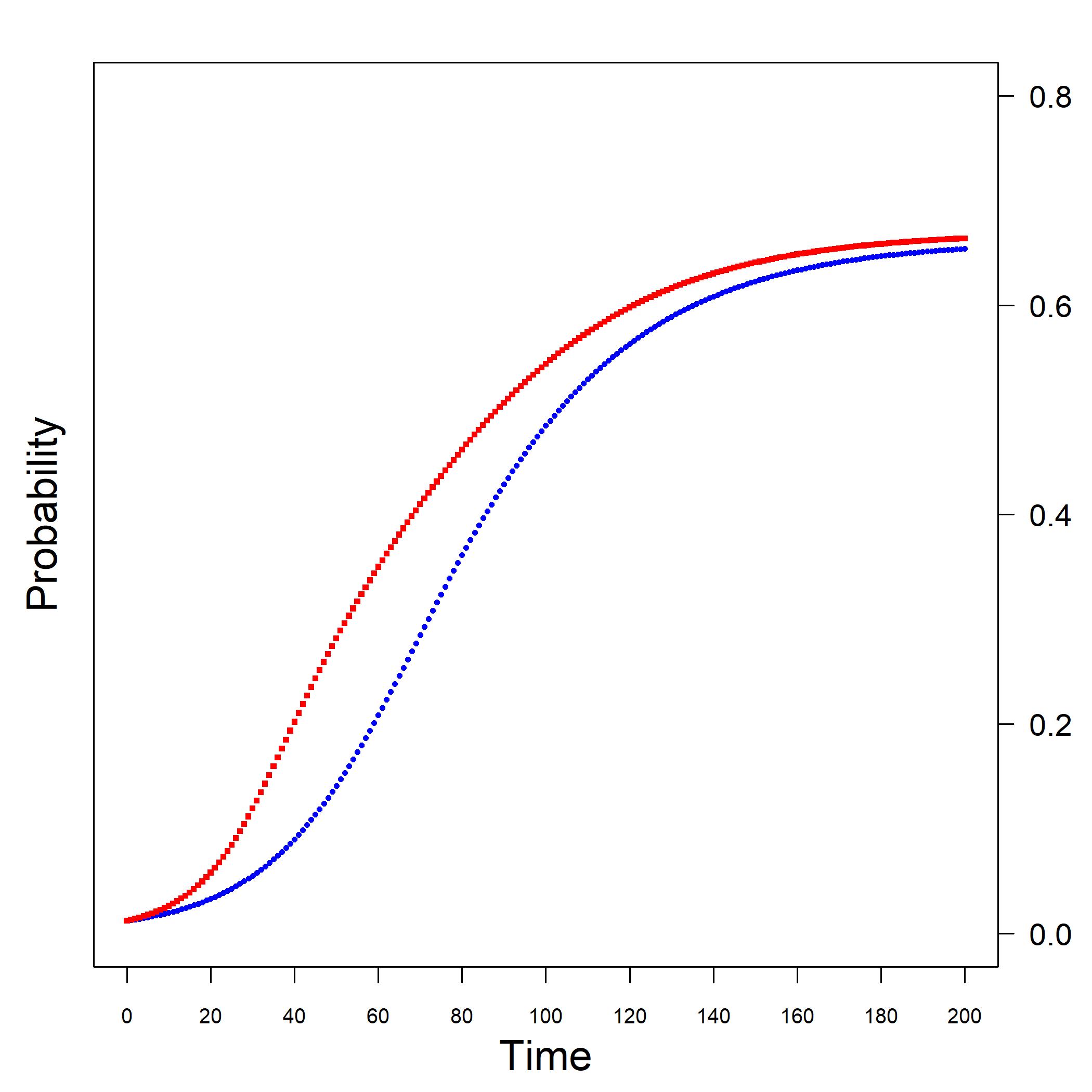}
		
	}\subfloat[]{\includegraphics[width=0.45\textwidth]{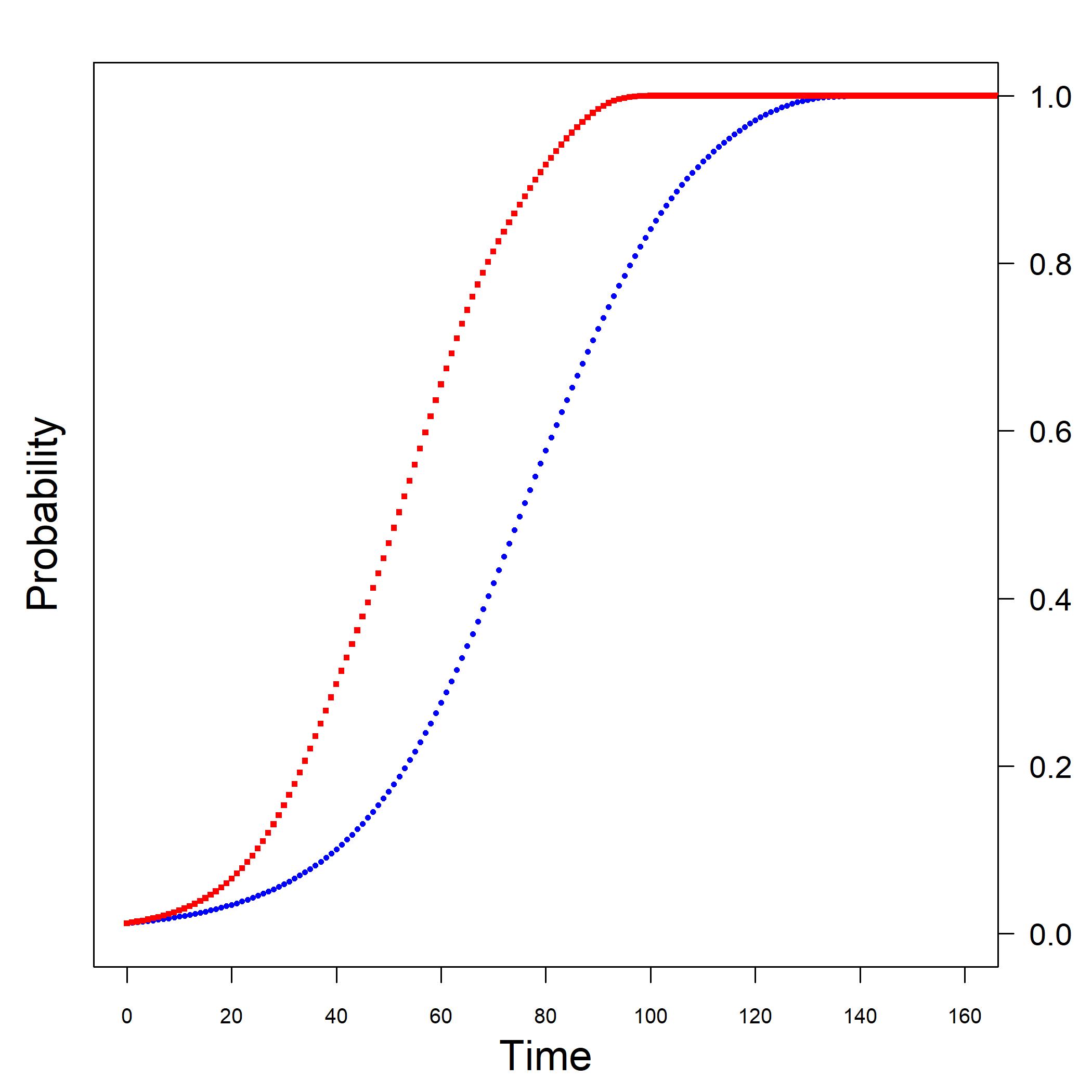}
		
	}

\caption{\textcolor{black}{Illustration of the time evolution of the cumulative
probability of getting infected obtained for the networks of hetero-
and homosexual contacts using the SIS model (a) and the Gompertz-like
solution obtained in this work (b). In panels (a) and (b) red circular
dots curve refers to the NHoC and blue square dots curve
to the NHeC. The time evolution of the cumulative
probability of two branches of the NHeC is obtained with the SIS model (c) and the Gompertz-like solution
obtained in this work (d). In panels (c) and (d) red circular dots
curve refers to branch 1 and blue square dots curve to branch 2 (see
the text for the involved nodes).}}

\label{Time_evolution} 
\end{figure}

\subsection{Temporal analysis of the probability density }

We now turn our attention to the time evolution of the population
density of infected individuals. As can be seen in Fig. \ref{sex}
(panel (a)), according to SIS, the growing period is practically the
same in both networks, with the NHoC reaching
a higher peak than the NHeC. Both networks reach their
maximum number of infected individuals almost at the same time. The
number of infected individuals per unit time decays slightly faster
in the NHoC than in the NHeC,
and both curves finally reach the zero percent of infected individuals
at the same time.

In the case of the Gompertz-like model (panel (b)) the differences
are more marked between the time evolution of the infection in both
networks. Here again the infection has the same growing process for
both networks, but from this time the differences of the dynamics
in the two networks are more significant. From $t=40$ to $t=75$
the infection decays faster in the NHeC. From that
time on, the decay is faster in the NHoC. Another important
characteristic observed in panel (b) is that while the whole decay
in the NHoC is smooth, it is irregular for the case
of the NHeC.

In order to further investigate the non smooth decay of the probability
density of infected individuals in the network of heterogeneous sexual
contacts we perform the same analysis by separating the nodes according
to the sex of the corresponding individuals. As can be seen in Fig.
\ref{sex} (panel (c) and (d)) both approaches (SIS and Gompertz, respectively)
predict a faster growing process for females than for males, with
a more pronounced difference in the Gompertz approach. However, while
the SIS approach predicts smooth decays of the number of infected
males and females, the Gompertz approach predicts an irregular decay
for males with two shoulders at different times, and a clearly non-monotonic
decay for females where a second peak is observed at about $t=80$.
%
%
%
%
%
%

\begin{figure}[h]
	\subfloat[]{\includegraphics[width=0.45\textwidth]{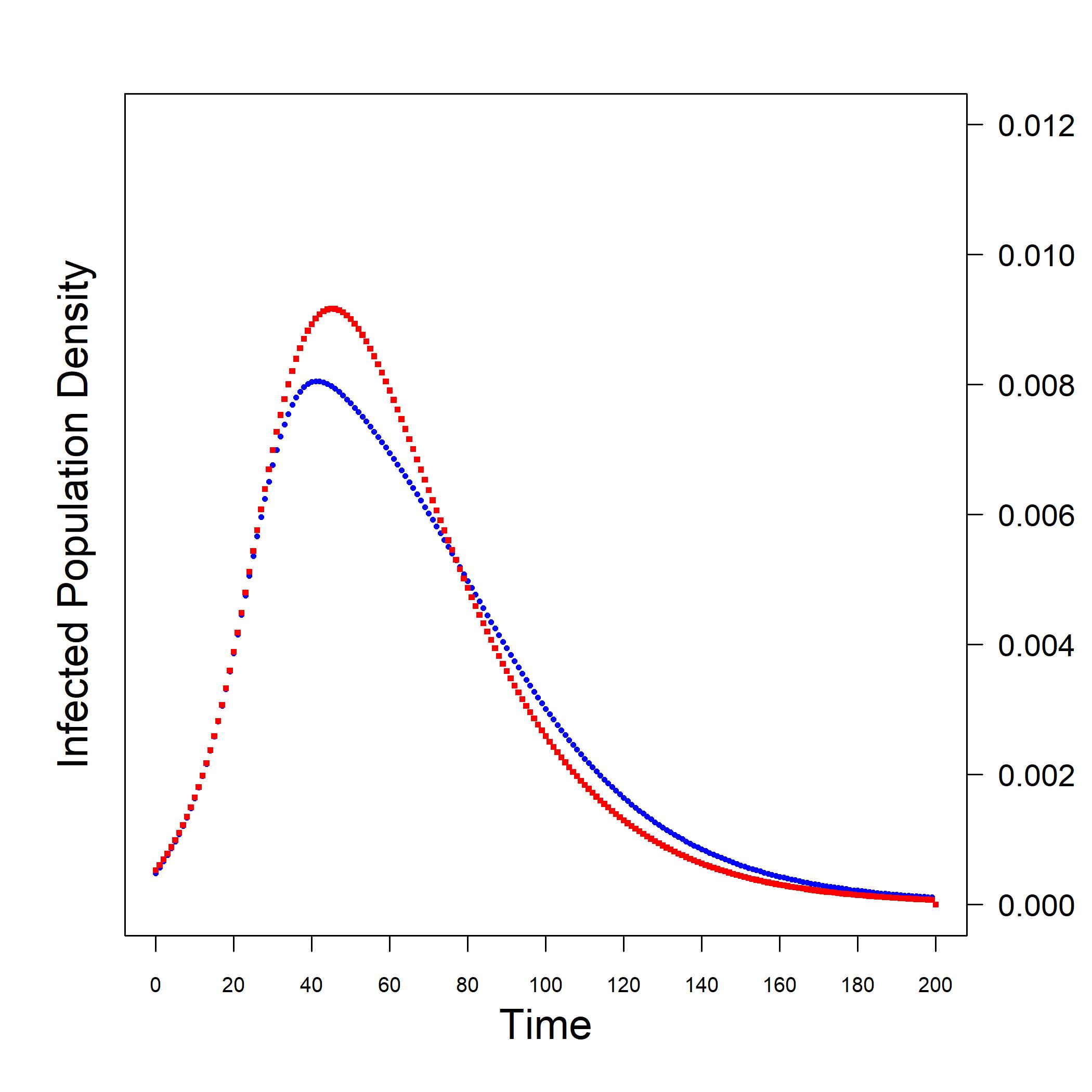}
		
	}\subfloat[]{\includegraphics[width=0.45\textwidth]{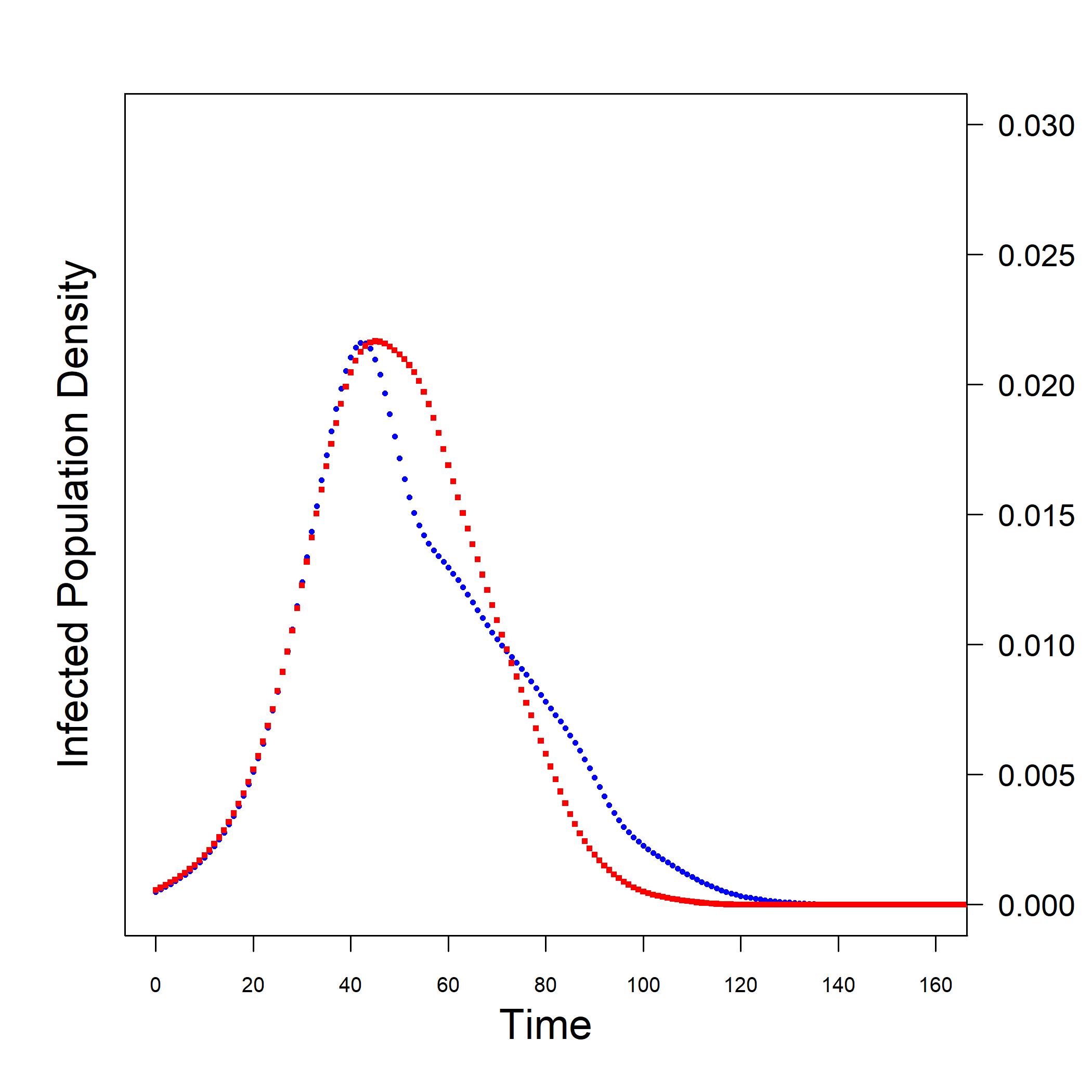}
		
	}
	
	\subfloat[]{\includegraphics[width=0.45\textwidth]{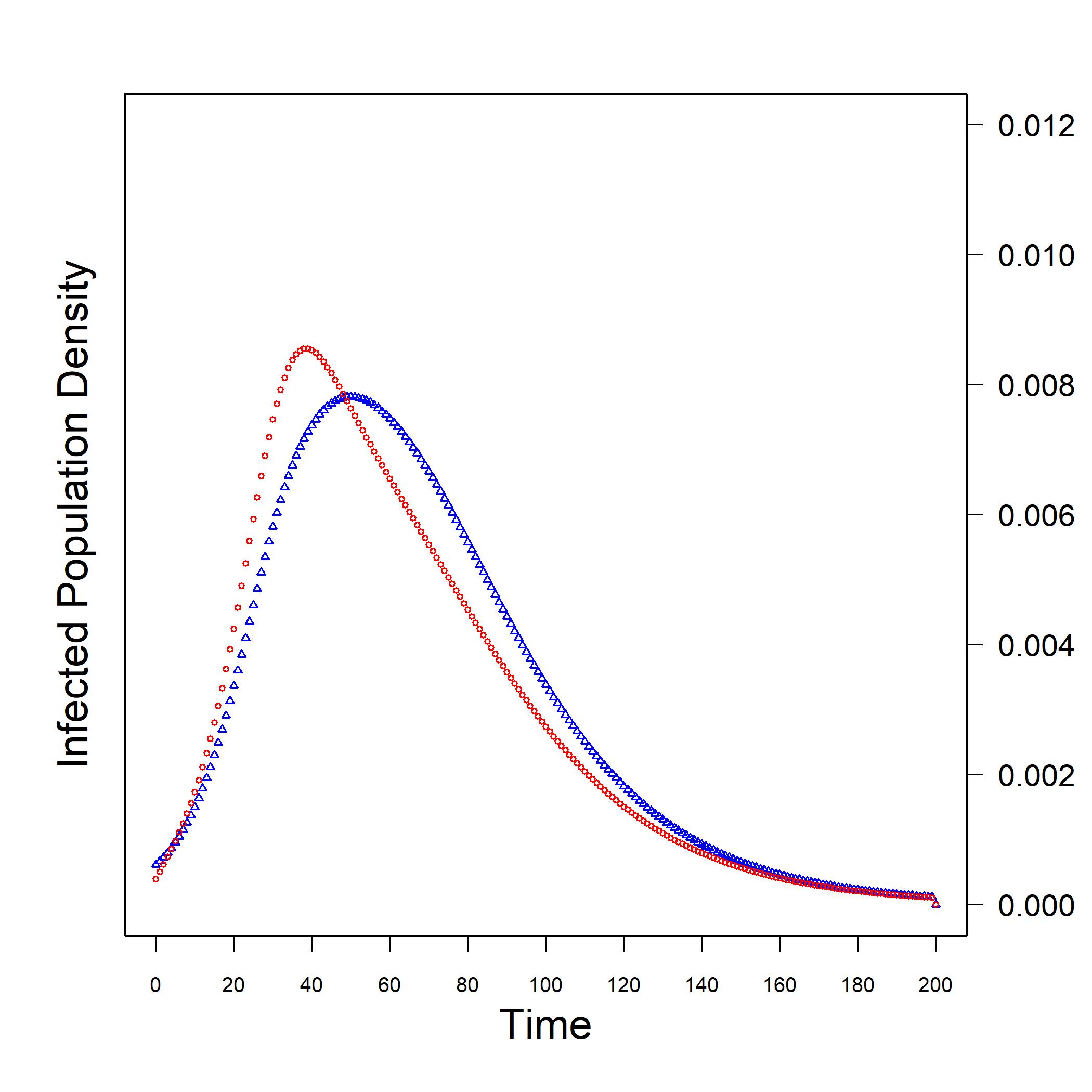}
		
	}\subfloat[]{\includegraphics[width=0.45\textwidth]{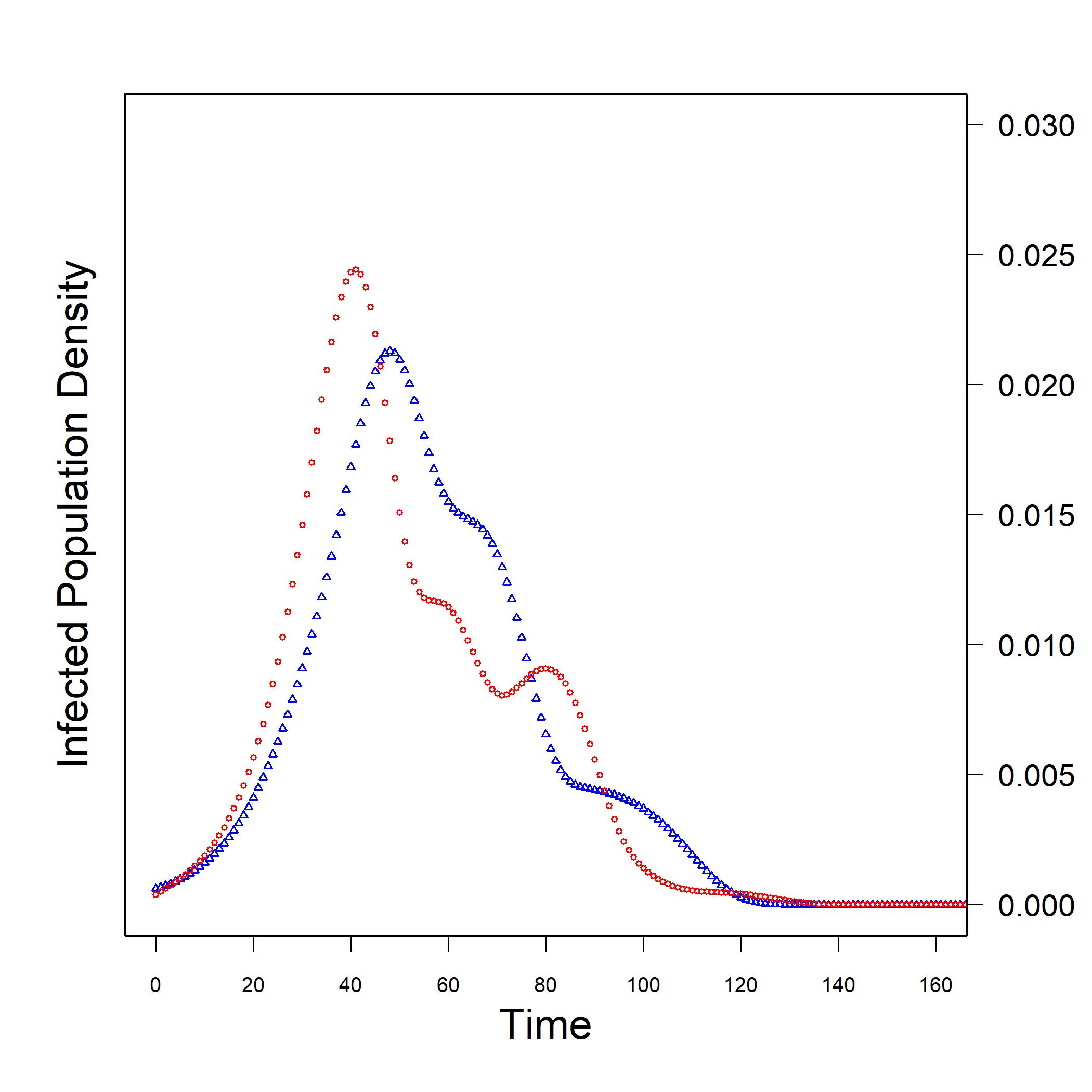}
		
	}

\caption{\textcolor{black}{Illustration of the time evolution of the probability
density of getting infected obtained for the networks of hetero- and
homosexual contacts using the SIS model (a) and the Gompertz-like
solution obtained in this work (b). In panels (a) and (b) red circular
dots curve refers to the NHoC and blue square dots curve
to the NHeC. The time evolution of the probability
density of males and females in the NHeC is
obtained with the SIS model (c) and the Gompertz-like solution obtained
in this work (d). In panels (c) and (d) red circles curve refers to
the female component and blue triangles curve to the male component.}}

\label{sex} 
\end{figure}

In the literature appear several reasons for this kind of ``multiple
peaks'' behavior in the dynamics of diseases. In the review paper
\citep{weiss2013sir} the author asks the question: ``What is the
mechanism(s) that is generating the multiple peaks? Nobody knows''.
In the context of Susceptible-Infected-Recovered (SIR) models the
authors of \citep{mummert2013perspective} proposed several alternatives,
two of which are applicable to the case studied here: (i) changes
in pathogen transmissibility and behavioral changes; (ii) population
heterogeneity where each wave spreads through one sub-population.
In another paper \citep{gupta1989networks} (see also \citep{anderson1991discussion}),
the authors reported that highly assortative mixing of the contacts
tends to lead to more rapid epidemic growth and can produce multiple
peaks in disease incidence. An explanation for the influence of assortative
mixing and multiple peaks is provided in \citep{hertog2007heterosexual}
by considering that the pathogen tends to move gradually from one
sexual activity group to the next, where the multiple peaks then correspond
to emerging infection within-group epidemics. A similar finding is
reported in \citep{xu2014long}. On the other hand, in \citep{perra2011towards}
the authors are able to generate multiple peaks of the infection by
tuning the parameters of a multi-compartment model.

We investigate here the causes of the multiple peaks observed when
using the networked Gompertz model. First we notice that the NHeC consists of two main blocks which corresponds
to individuals living in northern Manitoba and in Winnipeg. We then
select nodes representing female individuals from both blocks, namely
F10, F37, F1 and F4 from the block of Northern Manitoba (left in Fig. \ref{Networks} (a))
and F43 and F45 from the block of Winnipeg (right in Fig. \ref{Networks} (a)). In Fig. \ref{females evolution} we illustrate the results of the progression of the infection through these nodes. We have initiated the process by considering that every
node has exactly the same probability of getting infected, i.e. $p=1/82$.
However, as ``all roads lead to Rome'', node M9, which has the highest
degree of the network (degree 21), has a high risk of contagion at
earlier times of the propagation of the infection. Then, it is very
clear in the figure that among the nodes selected here, node F10,
which is connected to M9 and also has a high degree (degree
11), is the first to reach the peak of infection. Then, it comes F37,
which is a node separated by only three steps from M9. The process
continues by infecting nodes F1 and F4, which are five steps away
from M9, with the difference that the infection has two routes to
infect F1 (there is a square: F10-M4-F2-M5-F10) and only one to infect
F4. Then, the infection jumps to the second block and infects node
F43 separated from M9 by seven steps and then F45 separated by nine
steps from M9.


\begin{figure}[h]
	\begin{centering}
		\includegraphics[width=0.75\textwidth]{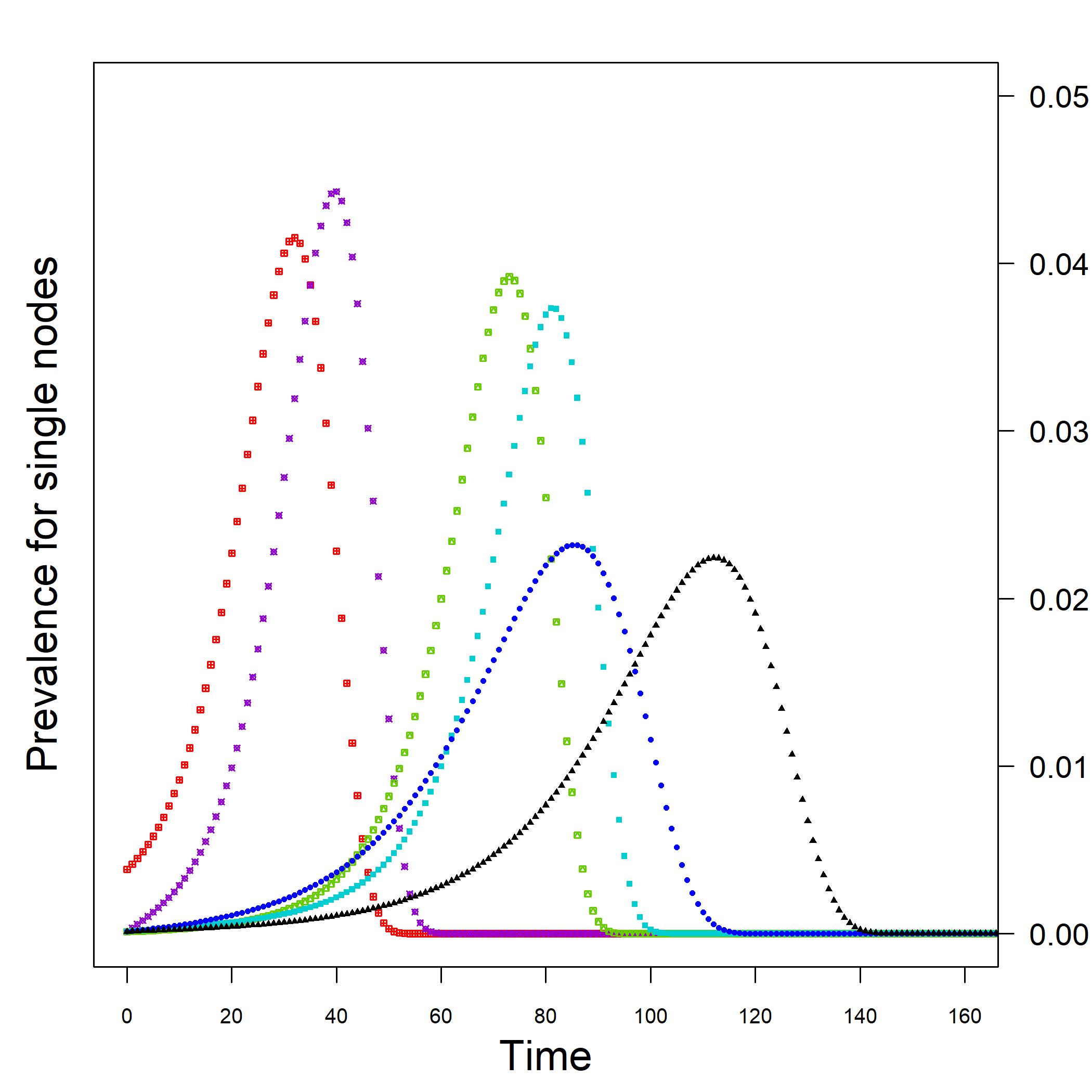} 
		\par\end{centering}
	\caption{\textcolor{black}{Illustration of the time evolution of the probability
			density of getting infected obtained for several nodes representing
			female individuals located at different ``regions'' of the network
			of heterosexual contacts (NHeC) using the Gompertz-like solution obtained
			in this work. From left to right: red F10; violet F37; green F1; cyan
			F4; blue F43; black F45.}}
\label{females evolution} 
\end{figure}

This result indicates that the networked Gompertz model obtained here
captures the wavefront nature of the propagation of a disease on a
network when more than one block or cluster exist. As we have seen
before this important feature of epidemic propagation is not captured
by the exact solution of the SIS model.
Multipeaks like the ones reproduced by the current model are observed in several real-life epidemics, ranging from the dengue fever outbreak in Havana (see Fig. 6 in \citep{weiss2013sir}) to the swine fever virus outbreaks in the Netherlands during 1997-1998 (see Fig. 7 in \citep{weiss2013sir} and \citep{Stegeman1999}).
Additionally, as the NHeC studied here is degree disassortative, our results show that the condition of assortativity is not necessary for the existence of these multipeaks as previously claimed in the literature \citep{gupta1989networks}.

\section{Conclusions}

It is known that the Gompertz function has a remarkable effectiveness in fitting experimental epidemiological and biological data, including the worldwide spread of COVID-19.  Here we approach the problem of combining this specific function with the information about the structure of real networks hosting a contagion process. 
To the best of our knowledge, a Gompertz function
on networks has not been proposed in the literature so far. This function
should take into account the topological structure of the network
and at the same time reduce asymptotically to the scalar Gompertz
function for sufficiently large times. The present paper fills this
gap, presenting a deduction of a Gompertz function on networks starting
from a classical contagion model such as the SIS model, which, as
is well known, is characterized by a symmetrical behavior with respect
to the inflection point. The search for an upper bound for the exact
solution to the SIS model on network has led to the identification of a new class of functions that
lack such symmetry but instead exhibit a typical Gompertz behavior.
We test this function on networks of sexual contacts that have already
been analyzed in the literature as the site of transmission for various
infectious diseases, which contemplate the possibility that a recovered
individual could become infected again. The main results can be summarized
as follows: 1) the function deduced as a worst-case scenario of the SIS model describes
effectively the behavior of the exact solution simulated on the network
under examination; 2) it allows to grasp the different behavior of
separate clusters in the network, such as the two components in a
bipartite network or specific subsets of nodes; 3) it allows to amplify
and therefore better highlight some behaviors reported in the literature
in the number of daily cases, such as the existence of multiple peaks
of subsequent infections, or the spread of infection from the core
towards the periphery of the network itself. Just as the scalar Gompertz
curve finds application in the description of an ever increasing number
of growth processes, we are confident that the function proposed here
could help in the description of similar processes hosted on specific
networks.

\section*{Acknowledgement}

E.E. thanks Grant PID2019-107603GB-I00 by MCIN/AEI/10.13039/501100011033.

 \bibliographystyle{elsarticle-num}
\bibliography{ref}

\end{document}